\pdfoutput=1
\RequirePackage{ifpdf}
\ifpdf 
\documentclass[pdftex]{sigma}
\else
\documentclass{sigma}
\fi

\numberwithin{equation}{section}

\newtheorem{Theorem}{Theorem}[section]
\newtheorem*{Theorem*}{Theorem}

\newtheorem{Lemma}[Theorem]{Lemma}

{ \theoremstyle{definition}
	\newtheorem{Definition}[Theorem]{Definition}
	
	\newtheorem{Example}[Theorem]{Example}
	\newtheorem{Remark}[Theorem]{Remark} }

\newcommand{\Real}{\mathbb R}
\newcommand{\N}{\mathbb N}
\newcommand{\ddbar}{\overline\partial}
\newcommand{\pr}{\partial}
\newcommand{\ol}{\overline}

\newcommand{\norm}[1]{\left\Vert#1\right\Vert}
\newcommand{\abs}[1]{\left\vert#1\right\vert}
\newcommand{\set}[1]{\left\{#1\right\}}
\newcommand{\To}{\rightarrow}

\begin{document}
	
\allowdisplaybreaks

\newcommand{\arXivNumber}{2410.03322}

\renewcommand{\PaperNumber}{048}

\FirstPageHeading

\ShortArticleName{Strict Quantization for Compact Pseudo-K\"ahler Manifolds and Group Actions}

\ArticleName{Strict Quantization for Compact Pseudo-K\"ahler\\ Manifolds and Group Actions}

\Author{Andrea GALASSO}

\AuthorNameForHeading{A.~Galasso}

\Address{Dipartimento di Matematica e Applicazioni, Universit\`a degli Studi di Milano Bicocca, \\ Via R.~Cozzi~55, 20125 Milano, Italy}
\Email{\href{mailto:andrea.galasso@unimib.it}{andrea.galasso@unimib.it}}

\ArticleDates{Received January 29, 2025, in final form June 17, 2025; Published online June 25, 2025}
	
\Abstract{The asymptotic results for Berezin--Toeplitz operators yield a strict quantization for the algebra of smooth functions on a given Hodge manifold. It seems natural to generalize this picture for quantizable pseudo-K\"ahler manifolds in presence of a group action. Thus, in this setting we introduce a Berezin transform which has a complete asymptotic expansion on the preimage of the zero set of the moment map. It leads in a natural way to prove that certain quantization maps are strict.}
	
	
\Classification{32V20; 32A25; 53D50}

\section{Introduction}
	
This paper is a continuation of \cite{gh}. Here, we further investigate the properties of the star product induced by Toeplitz operators, with a focus on establishing conditions under which certain quantization maps are strict. The present paper aims to extend known results for K\"ahler manifolds to the setting of pseudo-K\"ahler manifolds equipped with a group action. We~emphasize that the purpose of this introduction is to provide a broad overview of the results.
	
	 Consider a compact connected complex manifold $(M, J)$ of real dimension $2n$, where $J$ is an integrable complex structure. Let $\omega$ be a real, non-degenerate closed $2$-form on $M$, and assume that it satisfies the compatibility condition
		$
		\omega(JX, JY) = \omega(X, Y)$, $ \forall X,Y \in TM$.
		In this case, the triple $(M, \omega, J)$ is called a \emph{pseudo-K\"ahler manifold}. Unlike in the standard K\"ahler case, we do not assume that $\omega$ is positive definite. Instead, $\omega$ has constant signature $(n_-, n_+)$, meaning that at each point $m \in M$, the symmetric bilinear form $g_m(X,Y) := \omega_m(X, JY)$ has signature~${(n_-, n_+)}$.
		
		Let $(L, h_L)$ be a Hermitian line bundle over $M$, endowed with a Hermitian connection $\nabla^L$ whose curvature $R^L$ satisfies
		$
		R^L = -2\pi {\rm i} \omega$.

		\begin{Definition}
			For each positive integer $k$, the \emph{quantization space} $\mathcal{Q}_k(M)$ of $(M, \omega)$ with respect to the polarization $J$ is defined as the space of harmonic $(0,n_-)$-forms with values in the $k$-th tensor power $L^{\otimes k}$ of $L$,
\[
			\mathcal{Q}_k(M) := \ker \square^{(n_-)}_k,
\]
			where \smash{$\square^{(n_-)}_k$} is the Kodaira Laplacian acting on $L^{\otimes k}$-valued $(0,n_-)$-forms.
		\end{Definition}
		
		In the special case when $n_- = 0$ (i.e., when $g$ is positive definite), this construction recovers the classical geometric quantization: $\mathcal{Q}_k(M)$ coincides with the space $H^0\bigl(M, L^{\otimes k}\bigr)$ of holomorphic sections. Hence, this framework generalizes the K\"ahler case to pseudo-K\"ahler manifolds with indefinite metrics. There is a natural inner product on $\mathcal{Q}_k(M)$, we refer to Section~\ref{sec:for} for the definition and for more details.
		
		Let $\{\cdot ,\cdot\}$ denote the Poisson bracket on $C^\infty(M)$ induced by the symplectic form $\omega$.
		We now recall the definition of a star product.
	
	\begin{Definition} \label{def:star}
		A star product for the algebra $C^{\infty}(M)$ is given by the formal power series
		\[g\star h= \sum_{j=0}^{+\infty} C_{j}(g,h) \nu^{-j} \]
		such that $\star$ is an associative $\mathbb C[[\nu]]$-linear product, that is, $(g\star h) \star l=g\star (h\star l) ,$ for all $g, h, l\in C^{\infty}(M)$ and $C_{0}(g,h)=g\cdot h $, $ C_{1}(g,h)-C_{1}(h,g)= {\rm i} \{g, h\} $, for all $g, h\in C^{\infty}(M)$. The star product is said to be of \textit{Wick type} if and only if the function appearing as first argument is differentiated in holomorphic directions while the one appearing as second argument is differentiated in anti-holomorphic directions.
		
		In addition to the defining conditions, a star product may enjoy further properties, such as
		\begin{itemize}\itemsep=0pt
			\item[(1)] (Unit) Let $1$ be the constant function equals to $1$ everywhere, then $1\star f= f\star 1 =f$ for every smooth function on $M$.
			\item[(2)] (Parity) For every smooth functions $f$ and $g$ on $M$, then $\overline{f\star g}=\overline{f}\star \overline{g}$.
			\item[(3)] (Trace) Denote the trace on $\mathcal{Q}_k(M)$ by $\mathrm{Tr}_k$. It induces a $\mathbb{C}[[\nu]]$-linear map
			\[\mathrm{Tr} \colon\ {C}^{\infty}(M)[[\nu]]\rightarrow \nu^{-n} \mathbb{C}[[\nu]]\]
			such that
			$\mathrm{Tr}(f\star g)=\mathrm{Tr}(g\star f)$.
		\end{itemize}
\end{Definition}
	
Let $G$ be a compact connected Lie group having dimension $\dim_{\mathbb{R}}G=n_G$. Assume that $M$ admits a Hamiltonian and holomorphic action of $G$ with moment map $\Phi$. Suppose that $0$, in the dual of the Lie algebra, is a regular value for $\Phi$. Assume that the pseudo-metric $g$ restricted to the orbit through any $m\in \Phi^{-1}(0)$ is non-degenerate and the action of $G$ on $\Phi^{-1}(0)$ is free. Then the symplectic reduction $M_G:=\Phi^{-1}(0)/G$ is a quantizable pseudo-K\"ahler manifold, see, for example, \cite[Appendix~B]{rungi}. More precisely, let $\iota \colon\Phi^{-1}(0) \hookrightarrow M$ denote the inclusion map. Then the following holds:
	\begin{itemize}\itemsep=0pt
			\item[(1)] the topological quotient $M_G := \Phi^{-1}(0)/G$ is a smooth manifold of dimension $2n - 2n_G$, and the quotient map $\pi \colon \Phi^{-1}(0) \to M_G$ is a principal $G$-bundle;
			\item[(2)] there exist a unique pseudo-Riemannian metric $g_G$ and complex structure $J_G$ on $M_G$ such that
$
			\pi^* g_G = \iota^* g$, $ \pi^* J_G = \iota^* J$,
			and the 2-form $\omega_G := g_G(\cdot, J_G \cdot)$ defines a symplectic form on the quotient.
	\end{itemize}
 Furthermore, we assume that the symmetric bilinear form $g_G$ has signature $\bigl(n_-^G, n_+^G\bigr)$.

	 Suppose that the action of $G$ on $M$ lifts on $(L,h_L)$. Then there exists an unitary representation of $G$ on $\mathcal{Q}_k(M)$. Let $\mathcal{Q}_k(M)^G$ be the space of $G$-fixed vectors in $\mathcal{Q}_k(M)$. In Section~\ref{s:prelim}, for every smooth function $f$ on $M_G$, we define a $G$-invariant Toeplitz operators $T^G_k[f]$ acting on~$\mathcal{Q}_k(M)^G$. Our next result show that $G$-invariant Toeplitz operators induces a star product~$\star$ which defines a strict deformation quantization in the sense of \cite[Chapter 2]{l}. The following theorem is a consequence of Theorem~\ref{thm:quant} which is formulated further down.
	
\begin{Theorem} \label{cor}
There exists a unique formal star-product of Wick type
\[f\star g := \sum_{j=0}^{+\infty}\nu^j C_j(f, g) ,\qquad C_j(f, g)\ \text{lies in} \ C^{\infty}(M_G), \]
such that for every smooth functions $f$ and $g$ on $M_G$, for every $N\in \mathbb{N}$ we have with suitable constants $K_N(f, g)$, and for all $k\in \mathbb{N}$
\begin{equation} \label{eq:cond}
\left\lVert T^G_k[f] T^G_k[g]-\sum_{j=0}^{N-1} k^{-j} T^G_k[C_j(f,g)]\right \rVert=K_N(f, g) k^{-N} .
\end{equation}
	 	
Furthermore, the star product $\star$ satisfies the properties $(1)$, $(2)$, and $(3)$ of Definition {\rm\ref{def:star}}.
\end{Theorem}

To a certain extent the strategy to show Theorem~\ref{cor} is similar to the strategy done in \cite{schl1}, details are given in Section~\ref{sec:proofcor}. Now, if we put $G=\{e\}$, where $e$ is the identity element, we recover previous results in \cite{schl1}, and we generalize it to pseudo-K\"ahler manifolds. Furthermore, we refer to the surveys \cite{ks,schl2}, and \cite{w} for detailed discussions on the induced star products.
	
In order to better contextualize the aim of this work and to motivate our main results, we~recall some relevant developments in the literature. In~\cite{bms}, the authors studied Berezin--Toeplitz quantization on Hodge manifolds, their proofs are based on the theory of generalized Toeplitz structures developed by L.~Boutet de Monvel, V.~Guillemin, and J.~Sj\"ostrand in the framework of microlocal analysis, see \cite{boutet-guillemin}. In \cite{schl2}, additional properties of the Berezin--Toeplitz star product were established, and it was shown that the structure of Berezin--Toeplitz operators naturally leads to the strictness of certain quantization maps. The idea of considering pseudo-K\"ahler manifolds is suggested by results contained in~\cite{hsma,mm}, in particular see \cite[Theorem~1.7]{mm06} and \cite[Section~8.2]{mm}, where the authors worked out the asymptotic expansion for the Bergman kernel for $(0,q)$-forms under the assumption that the curvature is non-degenerate.
	
In the geometric quantization of pseudo-K\"ahler manifolds, two main strategies are available: one based on the theory of Berezin--Toeplitz quantization with vector bundles, as developed by X.~Ma and G.~Marinescu \cite{mm}, and another one employing Toeplitz operators on the associated circle bundle. In this work, we adopt the second approach. This choice is primarily motivated by the fact that it aligns with the framework we have developed in a previous article \cite{gh}.
	 	
The main approach of X. Ma and G. Marinescu is the local index theory, especially the analytic localization techniques developed by J.-M.~Bismut and G.~Lebeau. Our approach is based of Fourier integral operators of complex type as developed by A.~Melin and J.~Sj\"ostrand. It is worth noting that, in principle, both methods could be employed in the context of this article. Our method can be applied to study Berezin--Toeplitz star product for Reeb invariant smooth functions on non-degenerate CR manifolds with transversal and CR $\mathbb R$-action.

We now introduce the setting for our quantization commutes with reduction result in the pseudo-K\"ahler case. Suppose that the action of $G$ on $M$ lifts to an action on $(L, h_L)$. Then $(L, h_L)$ descends to a Hermitian line bundle $\bigl(L_G, h_L^G\bigr)$ over $M_G$ with curvature $R^{L_G} = -2\pi {\rm i} \omega_G$.
	Recall that the standard K\"ahler case is recovered when $n_-=0$, and consequently $n_-^G=0$.
		
This allows us to define the quantization of the reduced space as follows.
		
\begin{Definition}
The quantization space $\mathcal{Q}_k(M_G)$ is defined as the space of harmonic $\bigl(0,n_-^G\bigr)$-forms with values in \smash{$L_G^{\otimes k}$}
\[
\mathcal{Q}_k(M_G) := \ker \Box_k^{(n_-^G)},
\]
where \smash{$\square^{(n_-^G)}_k$} is the Kodaira Laplacian acting on \smash{$L^{\otimes k}_G$}-valued $\bigl(0,n_-^G\bigr)$-forms.
		\end{Definition}
		
There is a canonical map $\mathcal{Q}_k(M)^G \to \mathcal{Q}_k(M_G)$, and our main result is the following (the proof is given in Section~\ref{sec:proofof1.2}).
		
\begin{Theorem}[quantization commutes with reduction]\label{thm:qr=0}
Under the above assumptions, for every $k \in \mathbb{N}$, the canonical map
$\mathcal{Q}_k(M)^G \longrightarrow \mathcal{Q}_k(M_G)$ is an isomorphism of finite-dimensional Hilbert spaces.
\end{Theorem}
	
The principle that ``quantization commutes with reduction'' was originally formulated by V.~Guillemin and S.~Sternberg in the K\"ahler setting~\cite{gs}. The present result extends this principle to the pseudo-K\"ahler setting, of which the K\"ahler case is a particular instance.
		
In the context of general symplectic manifolds, proofs of this principle were developed by E.~Meinrenken \cite{14}, M.~Vergne \cite{v,v-II}, and E.~Meinrenken and R.~Sjamaar~\cite{16}. Independently, Y.~Tian and W.~Zhang \cite{tz} gave a different analytic approach. For a survey on the topic, we~refer to~\cite{ma}, where the connection with the CR setting introduced in \cite{hmm} is also highlighted.
		
Finally, let us also mention that in many of the aforementioned approaches, the quantization space is defined as the index of a twisted Spin$^c$-Dirac operator. In contrast, our proof of Theorem~\ref{thm:qr=0} follows the geometric strategy introduced in~\cite{gs}, adapted to the setting of indefinite curvature. We refer to Section~\ref{sec:quot} for the detailed construction and background.

\section{Background and statement of the results}	 \label{sec:for}
	
\subsection{Examples}
	 There are many examples of pseudo-K\"ahler manifolds that serve as natural candidates for quantization. First, we refer to the thesis \cite{rungi}, where the author investigates the symplectic and pseudo-Riemannian geometry of the $\mathbb{P}{\rm SL}(3,\mathbb{R})$ Hitchin component associated with a closed orientable surface. When the genus is at least~$2$, the author defines a mapping class group-invariant pseudo-K\"ahler metric on the Hitchin component.
	
	 A second example is given by the Grauert tube of an analytic pseudo-Riemannian manifold; see \cite{z} for an introduction in the classical Riemannian setting. Finally, we present a natural example arising in the context of the Heisenberg group.
	
	 Before describing this example, we recall some material from~\cite{fs}. A CR manifold $\bigl(X, T^{1,0}X\bigr)$, with contact form $\omega_0$, is said to be nondegenerate if its Levi form is nondegenerate at every point, for precise definitions, we refer the reader to Section~\ref{sec:crform}. In this case, replacing $\omega_0$ by~$-\omega_0$ if necessary, we may assume that the (constant) dimension $k$ of a maximal positive-definite subspace for the Levi form is at least $n/2$. We then say that $X$ is $k$-strongly pseudoconvex. When $k = n$, we simply say that $X$ is strongly pseudoconvex. If $X$ is $k$-strongly pseudoconvex with $k < n$, then by \cite[Lemma 13.1]{fs}, there exist smooth subbundles $E^+$ and $E^-$ of $T^{1,0}X$ such that $T^{1,0}X=E^+\oplus E^-$, the Levi form is positive definite on $E^+$, negative definite on $E^-$, and~${E^+ \perp E^-}$ with respect to the Levi form.
	
	 \begin{Example}
	 	Let $\mathbb{H}_n = \mathbb{C}^n \times \mathbb{R}$ denote the Heisenberg group. Occasionally, as in \cite{f}, it is useful to modify the group law on $\mathbb{H}_n^{{\rm pol}} = \mathbb{R}^{2n+1}$ via
	 	$(p, q, t)\cdot (p', q', t')=(p+p', q+q', t+t'+pq')$
	 	where~${z=p+{\rm i} q}$. Moreover, following \cite[pp.~436--437]{fs}, we may define a left-invariant CR structure on \smash{$\mathbb{H}_n^{{\rm pol}}$} which is $k$-strongly pseudoconvex, with $n/2 \leq k \leq n$, by taking \smash{$T^{1,0}\mathbb{H}_n^{{\rm pol}}$} to be the bundle spanned by
	 	$Z_1, \dots, Z_k, \overline{Z}_{k+1}, \dots, \overline{Z}_n $,
	 	where
	 	\smash{$Z_j=\frac{\partial}{\partial z_j}+{\rm i} \overline{z}_j \frac{\partial}{\partial t}$}.
	 	
	 	By compactifying this example as in \cite[p. 68]{f}, one obtains a natural compact pseudo-K\"ahler~manifold. Let $\Gamma$ be the subgroup of \smash{$\mathbb{H}_n^{{\rm pol}}$} consisting of points with integer coordinates
	 	\[\Gamma =\big\{(p, q, t)\in \mathbb{H}^{\rm pol}_n \mid p, q\in \mathbb{Z}^n\text{ and }t\in \mathbb{Z} \big\} . \]
	 	Then $\Gamma$ is a discrete subgroup of \smash{$\mathbb{H}_n^{{\rm pol}}$}, and the right coset space \smash{$M = \Gamma \backslash \mathbb{H}_n^{{\rm pol}}$} is a compact~mani\-fold. It is easy to verify that the half-open unit cube $Q^{2n+1}$ in \smash{$\mathbb{H}_n^{{\rm pol}}$} is a fundamental domain for~$\Gamma$, meaning that each right coset of $\Gamma$ contains precisely one point of $Q^{2n+1}$. Hence, topologically, $M$ can be viewed as the closed unit cube $Q^{2n+1}$ with certain boundary faces identified—a kind of ``twisted torus''. Moreover, by the above discussion, $M$ inherits a pseudo-K\"ahler structure. \end{Example}
	
	 \subsection{Toeplitz operators in the pseudo-K\"ahler setting}
	
Let $(M, \omega, J)$ be a {pseudo-K\"ahler manifold} and $\dim_{\mathbb{R}}M=2n$. Let $(L, h_L)$ be a Hermitian line bundle on $M$ having Hermitian connection whose curvature is $R^L=-2\pi {\rm i} \omega$. Fix a Hermitian metric $\langle \cdot | \cdot \rangle$ on the holomorphic tangent bundle $T^{1,0}M$. The Hermitian metric $\langle \cdot | \cdot \rangle$ on $T^{1,0}M$ induces a Hermitian metric on $TM\otimes\mathbb C$ and also on
 \[
 \bigoplus_{q\in\mathbb N\cup\set{0},\,0\leq q\leq n}T^{*0,q}M,
\]
 where $T^{*0,q}M$ is the bundle of $(0,q)$ forms. Consider the vector bundle $T^{*0,q}M\otimes L^{\otimes k}$ whose space of smooth sections is denoted by $\Omega^{0,q}\bigl(M,L^{\otimes k}\bigr)$.
	 	
	 The Hermitian metrics $\langle \cdot | \cdot \rangle$ and $h^L$ induce a new Hermitian metric $\langle \cdot | \cdot \rangle_{h^{L^{\otimes k}}}$ on $T^{*0,q}M\otimes L^{\otimes k}$, the corresponding norm is denoted by \smash{$|\cdot|_{h^{L^{\otimes k}}}$}. We can define the $L^2$-inner product as follows
	 \[(s_1,s_2)_k=\int_M \langle s_1 | s_2 \rangle_{h^{L^{\otimes k}}} \mathrm{dV}_M,\qquad s_1, s_2\in \Omega^{0,q}\bigl(M,L^{\otimes k}\bigr) , \]
	 and we denote by $\lVert \cdot \rVert_{k}$ the corresponding norm, where $\mathrm{dV}_M$ is the volume form on $M$ induced by $\langle \cdot | \cdot \rangle$.
	 Let $L^2_{0,q}\bigl(M,L^{\otimes k}\bigr)$ be the completion of $\Omega^{0,q}\bigl(M,L^{\otimes k}\bigr)$ with respect to $(\cdot ,\cdot)_k$. Let
 \[
 \square^{(q)}_k\colon\ \Omega^{0,q}\bigl(M,L^{\otimes k}\bigr)\To\Omega^{0,q}\bigl(M,L^{\otimes k}\bigr)
 \]
 be the Kodaira Laplacian. We denote by the same symbol the $L^2$ extension of $\square^{(q)}_k$
	 \[\square^{(q)}_k\colon\ \operatorname{Dom}\square^{(q)}_k\subset L^2_{0,q}\bigl(M,L^{\otimes k}\bigr)\To L^2_{0,q}\bigl(M,L^{\otimes k}\bigr),\]
	 where
\smash{$
\operatorname{Dom}\square^{(q)}_k=\big\{u\in L^2_{0,q}\bigl(M,L^{\otimes k}\bigr)\mid \square^{(q)}_ku\in L^2_{0,q}\bigl(M,L^{\otimes k}\bigr)\big\}$}.
	 The projection
	 \[ P^{(q)}_k \colon\ L^2_{0,q}\bigl(M,L^{\otimes k}\bigr) \rightarrow H^q\bigl(M,L^{\otimes k}\bigr):=\operatorname{Ker}\square^{(q)}_k
\]
	 onto the kernel of \smash{$\square^{(q)}_k$} is called Bergman projector or Bergman projection. Since \smash{$\square^{(q)}_k$} is elliptic, its distributional kernel satisfies
	 \[P^{(q)}_k( \cdot , \cdot )\in\mathcal{C}^{\infty}\bigl(M\times M, L^k\otimes \bigl(\bigl(T^{*0,q}M\bigr)\boxtimes \bigl(T^{*0,q}M\bigr)^*\bigr)\otimes \bigl(L^k\bigr)^*\bigr) . \]

Now, we shall follow \cite{hsma1} and define Toeplitz operators in this framework. Fix a smooth function $f$ on $M$. Define the Berezin--Toeplitz operator
	 \[T^{(q)}_{k}[f]:= P^{(q)}_k\circ f \circ P^{(q)}_k \colon\ L^2_{0,q}\bigl(M,L^{\otimes k}\bigr) \rightarrow H^q\bigl(M,L^{\otimes k}\bigr):=\operatorname{Ker}\square^{(q)}_k ; \]
	 we always fix $q=n_-$ and we write \smash{$T_{k}[f]=T^{(n_-)}_{k}[f]$}. Note that if $M$ is compact complex manifold endowed with a positive line bundle $L$, we put $q=n_-=0$ and we recover the standard K\"ahler setting. Thus, the quantizing Hilbert space is \smash{$\operatorname{Ker}\square^{(q)}_k$}, for Riemann--Roch-type results we refer to \cite{hsma1} and \cite[Section~1.2]{gh2}.

 	\subsection{CR formulation} \label{sec:crform}
 		
	The circle bundle $X$ inside the dual of the line bundle $L$ is a compact, connected and orientable CR manifold of dimension $2n+1$. We now provide a more precise formulation of these concepts. Let $\bigl(X, T^{1,0}X\bigr)$ be a compact, connected and orientable CR manifold of dimension $2n+1$, $n\geq 1$, where $T^{1,0}X$ is a subbundle of the complexified tangent bundle $ TX\otimes \mathbb{C}$. There is a~unique sub-bundle $HX$ of $TX$ such that
$
 HX\otimes \mathbb{C}=T^{1,0}X \oplus T^{0,1}X
$
 which is called the horizontal bundle. Let~${J \colon HX\rightarrow HX}$ be the complex structure map given by $J(u+\ol u)={\rm i}u-{\rm i}\ol u$, for every~${u\in T^{1,0}X}$. By complex linear extension of $J$ to $TX\otimes \mathbb{C}$, the ${\rm i}$-eigenspace of $J$ is $T^{1,0}X$.
	
	Since $X$ is orientable, there always exists a real non-vanishing $1$-form $\omega_0\in {C}^{\infty}(X,T^*X)$ so that $\langle \omega_0(x) , u \rangle=0$, for every $u\in H_xX$, for every $x\in X$. The one form $\omega_0$ is called contact form and it naturally defines a volume form $\mathrm{dV}_X$ on $X$. For each $x \in X$, we define a quadratic form on $HX$ with respect to $\omega_0$ by
	${L}_x(U,V) =\frac{1}{2}{\rm d}\omega_0(JU, V)$, $ \forall U, V \in H_xX$.
	Then, we extend~${L}$ to~${HX\otimes \mathbb{C}}$ by complex linear extension; for \smash{$U, V \in T^{1,0}_xX$},
	\[
	{L}_x\bigl(U,\overline{V}\bigr) = \frac{1}{2} {\rm d}\omega_0\bigl(JU, \overline{V}\bigr) = -\frac{1}{2{\rm i}} {\rm d}\omega_0\bigl(U,\overline{V}\bigr).
	\]
	The Hermitian quadratic form ${L}_x$ on $T^{1,0}_xX$ is called Levi form at $x$ (with respect to $\omega_0$). Hereafter we shall always assume that the Levi form is non-degenerate with $n_-$ negative and~$n_+$ positive eigenvalues, $n_++n_-=n$. The Reeb vector field $R\in{C}^\infty(X,TX)$ is defined to be the non-vanishing vector field determined by
$\omega_0(R)\equiv -1$, $ 		{\rm d}\omega_0(R,\cdot)\equiv0 $ on $ TX$.
	There is a Reeb vector field such that the flow of it induces a transversal CR $\mathbb{R}$-action on~$X$. Suppose that all orbits of the flow of $R$ are compact. Then $X$ admits a transversal CR circle action, which we denote by ${\rm e}^{{\rm i} \theta}\cdot$, where $\theta\in \mathbb{R}$. Note that this is the case of the unit circle bundle~${X\subset L^{\vee}}$, with projection $\pi\colon X\rightarrow M$ and connection contact form $\omega_0$ such that~${{\rm d}\omega_0 =2 \pi^*(\omega)}$. Then $(X, \omega)$ is a CR manifold and $\omega_0$ is the contact form.
	
	Now, we define the quantizing space and we identify it later with the kernel of the Kodaira Laplacian. For all $x\in X$, let $T^{*0,q}X$ be the $q$-th exterior power of $\bigl(T^{(0,1)}X\bigr)^*$, whose space of sections is denoted by $\Omega^{0,q}(X)$ and its elements are called $(0,q)$-forms. The Hermitian metric~$\langle \cdot | \cdot \rangle$ on $TX\otimes \mathbb{C}$ induces in a natural way a $L^2$ inner product $( \cdot | \cdot )$ on $\Omega^{0,q}(X)$. Denote as $\lVert \cdot\rVert$ the corresponding norm of $( \cdot | \cdot )$. Let \smash{$L^2_{(0,q)}(X)$} be the completion of $\Omega^{0,q}(X)$ with respect to $( \cdot | \cdot )$. Let $\Box^q_b$ be the Gaffney extension of the Kohn Laplacian. The Szeg\H{o} projection is the orthogonal projection
\smash{$S^{(q)} \colon L^2_{(0,q)}(X)\To \operatorname{Ker}\Box^q_b$}
	with respect to $( \cdot | \cdot )$. Let
\[
S^{(q)}(x,y)\in {D}'\bigl(X\times X,T^{*0,q}X\boxtimes\bigl(T^{*0,q}X\bigr)^*\bigr)
\]
 be the distributional kernel of $S^{(q)}$, ${D}'$ denotes the space of distribution sections, see Section~\ref{sec:crmanif} for notation. In this paper, we will assume that $\Box^q_b$ has $L^2$ closed range. This hypothesis is always satisfied for the circle bundle case we are considering.
	
	Fix $q=n_-$ and let $f\in C^{\infty}(X, \mathbb{C})$, then the Toeplitz operator is given by
	\[T^X[f]:=S^{(n_-)}\circ f\cdot\circ S^{(n_-)}\colon\ L^2_{(0,n_-)}(X)\To\operatorname{Ker}\Box^{n_-}_b .
	\]
	The transversal circle action can be used to define the corresponding Fourier components which we denote by $T_k^X[f]$, see \cite{gh} where we use pseudodifferential operators of order zero (furthermore we refer to \cite{sh} for composition with certain classical pseudodifferential operators of order one). Let $u\in\Omega^{0,q}(X)$ be arbitrary. Define
	\[
	Ru:=\frac{\pr}{\pr\theta}\bigl(\bigl({\rm e}^{{\rm i}\theta}\bigr)^*u\bigr)|_{\theta=0}\in\Omega^{0,q}(X).
	\]
	For every $k\in\mathbb Z$, let
	\[
	\Omega^{0,q}_k(X):=\set{u\in\Omega^{0,q}(X)\mid Ru={\rm i}k u},\qquad q=0,1,2,\dots,n .\]
	Thus, we denote \smash{${C}^\infty_k(X):=\Omega^{0,0}_k(X)$}. Then, the $L^2$ inner product $( \cdot | \cdot )$ on $\Omega^{0,q}(X)$
	induced by~$\langle \cdot | \cdot \rangle$ is $S^1$-invariant. Let $L^2_{(0,q), k}(X)$ be
	the completion of $\Omega_k^{0,q}(X)$ with respect to $( \cdot | \cdot )$.
	The $k$-th Szeg\H{o} projection is the orthogonal projection
\smash{$S^{(q)}_{k}\colon L^2_{(0,q)}(X)\To \bigl(\operatorname{Ker}\Box^q_b\bigr)_k$}
	with respect to~$( \cdot | \cdot )$. For ease of notation we put
	\smash{$H(X):=\operatorname{Ker}\Box^{n_-}_b$} and \smash{$H_k(X):=\bigl(\operatorname{Ker}\Box^{n_-}_b\bigr)_k$}.
	Eventually, for every $S^1$ invariant function $f\in{C}^\infty(X)$, the $k$-th Toeplitz operator is given by
	\[
		T_{k}^X[f]:=S^{(n_-)}_{k}\circ f\cdot\circ S^{(n_-)}_{k}\colon\ L^2_{(0,n_-)}(X)\To H_k(X) .
	\]
	
	Now, recall that $X$ is the circle bundle inside the dual of the line bundle $(L, h_L)$. We have defined two important operators: $T_{k}^X[f]$ \big(constructed with the $k$-th Fourier component of the Szeg\H{o} projector \smash{$S^{(n_-)}_{k}$}\big) and $T_{k}[f]$ \big(constructed with the Bergman projector \smash{$P_{k}^{(n_-)}$}\big). They can be identified, in fact we claim
	\[S_{k}^{(n_-)}(x_m,x'_{m'})=(x_m)^k \left(\frac{1}{2\pi}P_{k}^{(n_-)}(m,m')(x'_{m'})^k \right) , \]
	where $\pi(x_m)=m$, $\pi(x'_{m'})=m'$.
	
	\begin{proof}[Proof of the claim]
		We adopt the following conventions and notations: every element $x^{\vee}$ in~$ {(L^{\vee})^{\vee}=L}$ defines a map $x^* \colon L^{\vee}\rightarrow \mathbb{C}$ which is the dual of an element $x\in L^{\vee}$. We also write~$f_m$ for the restriction of \smash{$f\in L^2_{(0,n_-),k}(X)$} along the fibre $S_m$ of the standard circle action on $X$. First, notice that any function \smash{$f\in L^2_{(0,n_-),k}(X)$} is uniquely defined at $m$ along the fibre~${S_m\subset L^{\vee}}$ by~\smash{$f\bigl({\rm e}^{-{\rm i}\theta}\cdot x_m\bigr)={\rm e}^{{\rm i}k\theta} f(x_m) $};
		and thus it defines a linear map ${f_m \colon (L^{\vee}_m)^{\otimes k}\rightarrow T^{*0,q}X}$, which is the same as an element $\sigma(m) \in T^{*0,q}_mM\otimes L^{\otimes k}_m$.
		We have the following isomorphism:
		\begin{align*}\mathfrak{S}_k \colon\ L^2_{(0,n_-),k}(X) \rightarrow L^2_{0,n_-}\bigl(M,L^{\otimes k}\bigr) ,\qquad
			f\mapsto \sigma_f\end{align*}	
		where the inverse of $\mathfrak{S}_k$ is given by $\sigma \mapsto f_{\sigma}$, with $f_{\sigma}(x_m)=(x_m)^k(\sigma(m))$. Finally, the operator given by taking the $k$-th Fourier component along the fibers $S$ of $\pi \colon X\rightarrow M$ is given explicitly~by
		\begin{align*}F_k \colon\ L^2_{(0,n_-)}(X) \rightarrow L^2_{(0,n_-),k}(X),\qquad
			f\mapsto \frac{1}{2\pi}\int_S f(x) (x^*)^k {\rm d}x .\end{align*}
		Now, the claim follows by the formula
		\smash{$S_{k}^{(n_-)}=\mathfrak{S}^{-1}_k\circ P_{k}^{(n_-)} \circ \mathfrak{S}_k\circ F_k$}.	
	\end{proof}

\subsection{Group action framework}

A Hamiltonian and holomorphic action of a compact Lie group $G$ on $M$ induces an infinitesimal action of its Lie algebra on $X$ via the Kostant formula, see \cite{ko}. Under suitable topological conditions—such as when $G$ is semisimple or the manifold $M$ is simply connected—this infinitesimal action can be integrated to an action of $G$ on the circle bundle $X$ inside the dual line bundle~$L^*$. This action of $G$ on $X$ preserves both $\omega_0$ and $J$. We are thus led to introduce the main definitions in the CR setting, but before doing so we give more details on the Kostant formula and the quantization of group actions on pseudo-K\"ahler manifolds.

Let $M$ be a pseudo-K\"ahler manifold with complex structure $J$ and Hodge form $\omega$. Recall that we denote by $X$ the circle bundle lying in the dual line bundle $L^*$, with circle action~${{\rm e}^{{\rm i}\theta}\cdot \colon S^1\times X\rightarrow X}$, projection $\pi\colon X\rightarrow M$ and $\partial_\theta$ the generator of the structure circle action on it. Hence $X$ is naturally a contact and Cauchy--Riemann manifold; $\omega_0$ is the contact form.

Suppose given, in addition, an action $\mu \colon G\times M\rightarrow M$ of a compact Lie group $G$ with Lie algebra $\mathfrak{g}$, which is holomorphic with respect to the complex structure $J$ and Hamiltonian with moment map $\Phi\colon M\rightarrow \mathfrak{g}^\vee$,
$2 \omega(\xi_M, \cdot)={\rm d}\Phi^{\xi}(\cdot)$.
By~\cite{ko}, the action $\mu$ naturally induces an infinitesimal contact action of $\mathfrak{g}$ on the circle bundle $X$. Explicitly, if $\xi\in \mathfrak{g}$ and $\xi_M$ is the corresponding Hamiltonian vector field on $M$, then its contact lift $\xi_X$ is as follows.
Let $\mathbf{v}^\sharp$ denote the horizontal lift on $X$ of a vector field $\mathbf{v}$ on $M$, we have
\begin{equation}\label{eqn:contact vector field}
\xi_X:= \xi_M^\sharp-\langle \Phi\circ \pi, \xi\rangle R .
\end{equation}

Furthermore, suppose that the infinitesimal action~\eqref{eqn:contact vector field} can be integrated to an action of~$G$ on~$X$. By hypothesis, it preserves the Cauchy--Riemann structure and the contact form $\alpha$. There is a naturally induced unitary representation of $G$ on the Hardy space
\smash{$H(X)\subset L^2_{(0, n_-)}(X)$} given by
$ g \colon f \mapsto f\circ {g^{-1}}\cdot$.
	
Thus, adopting the convection of the previous subsection, we find that $X$ admits an action of a compact connected Lie group $G$ of dimension $d$, and the $G$ action preserves $\omega_0$ and the CR structure. Furthermore, for any $\xi \in \mathfrak{g}$, the vector field $\xi_X$ appearing in equation \eqref{eqn:contact vector field} is~\smash{$(\xi_X u)(x)=\frac{\partial}{\partial t}\bigl[u\bigl({\rm e}^{t\xi} x\bigr)\bigr]_{|_{t=0}}$}, for any $u\in {C}^\infty(X)$.
	
Fix a smooth Hermitian metric $\langle \cdot | \cdot \rangle$ invariant under the action of $G$ and $S^1$ on $TX\otimes \mathbb{C}$ so that $T^{1,0}X$ is orthogonal to $T^{0,1}X$, $\langle u | v \rangle$ is real if $u, v$ are real tangent vectors, and $\langle R | R \rangle=1$. Now, we define the quantizing $G$-invariant space. Fix $g\in G$ and $x\in X$, for any $r\in \mathbb{N}$ let
\[
g^*\colon \
\Lambda^r_x(T^*X\otimes \mathbb{C})\To\Lambda^r_{g^{-1}\circ x}( T^*X\otimes \mathbb{C})
\]
 be the pullback map. Since $G$ preserves $J$, we have
\smash{$
g^*\colon T^{*0,q}_xX\To T^{*0,q}_{g^{-1}\circ x}X
$}
 for all $x\in X$. Thus, for $u\in\Omega^{0,q}(X)$, we have $g^*u\in\Omega^{0,q}(X)$. Put
 \[\Omega^{0,q}(X)^G:=\big\{u\in\Omega^{0,q}(X) \mid g^*u=u,\,\forall g\in G\big\} .
 \]
	Since the Hermitian metric $\langle \cdot | \cdot \rangle$ on $TX\otimes \mathbb{C}$ is $G$-invariant, the $L^2$ inner product $( \cdot | \cdot )$ on $\Omega^{0,q}(X)$
	induced by $\langle \cdot | \cdot \rangle$ is $G$-invariant. Denote as $\lVert \cdot\rVert$ the corresponding norm of $( \cdot | \cdot )$. Let $u\in L^2_{(0,q)}(X)$ and $g\in G$, we define $g^*u$ in the standard way. Put
	\[L^2_{(0,q)}(X)^G:=\big\{u\in L^2_{(0,q)}(X)\mid g^*u=u,\,\forall g\in G\big\}.\]
	Put
\smash{$\bigl(\operatorname{Ker}\Box^q_b\bigr)^G:=\operatorname{Ker}\Box^q_b\cap L^2_{(0,q)}(X)^G$}. The $G$-invariant Szeg\H{o} projection is the orthogonal projection
	\[S^{(q)}_G\colon \ L^2_{(0,q)}(X)\To \bigl(\operatorname{Ker}\Box^q_b\bigr)^G\]
	with respect to $( \cdot | \cdot )$. Let \smash{$S^{(q)}_G(x,y)\in {D}'\bigl(X\times X,T^{*0,q}X\boxtimes\bigl(T^{*0,q}X\bigr)^*\bigr)$} be the distributional kernel of \smash{$S^{(q)}_G$}, ${D}'$ denotes the space of distribution sections, see Section~\ref{sec:crmanif} for notation.
	
	Let $f\in C^{\infty}(X, \mathbb{C})$ be invariant under the circle action and the action of $G$, then the $G$-invariant Toeplitz operator is given by
	\[T^G[f]:=S^{(n_-)}_G\circ f\cdot\circ S^{(n_-)}_G\colon\ L^2_{(0,n_-)}(X)\To\bigl(\operatorname{Ker}\Box^{n_-}_b\bigr)^G .
	\]
	The transversal circle action can be used to define the corresponding Fourier components. First, we note that
$
	{\rm e}^{{\rm i}\theta}\cdot (g x)=g \bigl( {\rm e}^{{\rm i}\theta}\cdot x\bigr)$, for every $ x\in X$, $ \theta\in[0,2\pi]$, $ g\in G$.
	For every $k\in\mathbb Z$, let
	\[\Omega^{0,q}_k(X)^G=\big\{u\in\Omega^{0,q}(X)^G \mid Ru={\rm i}k u\big\},\qquad q=0,1,2,\dots,n.
	\]
	Thus, we denote \smash{${C}^\infty_k(X)^G:=\Omega^{0,0}_k(X)^G$}. Then, the $L^2$ inner product $( \cdot | \cdot )$ on $\Omega^{0,q}(X)$
	induced by~$\langle \cdot | \cdot \rangle$ is $G$ and $S^1$-invariant. Let \smash{$L^2_{(0,q), k}(X)^G$} be
	the completion of $\Omega_k^{0,q}(X)^G$ with respect to~$( \cdot | \cdot )$.
	The $k$-th $G$-invariant Szeg\H{o} projection is the orthogonal projection
	\[S^{(q)}_{G,k}\colon\ L^2_{(0,q)}(X)\To \bigl(\operatorname{Ker}\Box^q_b\bigr)^G_k\]
	with respect to $( \cdot | \cdot )$. For ease of notation, we put
	\smash{$H^G(X):=\bigl(\operatorname{Ker}\Box^{n_-}_b\bigr)^G$}
	and
$H^G_k(X):=\smash{\bigl(\operatorname{Ker}\Box^{n_-}_b\bigr)^G_k}$.
	Eventually, for every $G$ and $S^1$ invariant function $f\in{C}^\infty(X)$, the $k$-th $G$-invariant Toeplitz operator is given by
	\[
		T^G_{k}[f]:=S^{(n_-)}_{G,k}\circ f\cdot\circ S^{(n_-)}_{G,k}\colon\ L^2_{(0,n_-)}(X)\To H^G_k(X) .
	\]
	
	 \subsection{A key result on Toeplitz operators}
	
	 The goal of this subsection is to establish a key result concerning Toeplitz operators. This result will serve as the main analytical tool in the proof of the principal theorem of the paper.
	
	 \begin{Definition}\label{d-gue170124}
	 	The contact moment map associated to the form $\omega_0$ is the map $\mu\colon X \to \mathfrak{g}^*$ such that, for all $x \in X$ and $\xi \in \mathfrak{g}$, we have
$\langle \mu(x), \xi \rangle = \omega_0(\xi_X(x))$.
	 \end{Definition}

 	Note that $\mu$ is $\Phi\circ \pi \colon X \rightarrow \mathfrak{g}^*$. Let $b$ be the nondegenerate bilinear form on $HX$ such that~${
	 	b(\cdot , \cdot) = {\rm d}\omega_0(\cdot , J\cdot)}$.
	 Clearly $b$ is the pullback via $\pi$ of the pseudo-Riemannian metric on $M$. Let us denote $\mu^{-1}(0)$ by $Y$. Assume that $0$ is a regular value of $\mu$, assume that the action of $G$ on $\mu^{-1}(0)$ is free and
	 \begin{equation}\label{e-gue200120yydI}
	\underline{\mathfrak{g}}_x\cap \underline{\mathfrak{g}}^{\perp_b}_x=\set{0}\qquad \text{at every point} \ x\in Y,
	 \end{equation}
	 where $\underline{\mathfrak{g}}={\rm Span}(\xi_X; \xi\in\mathfrak{g})$
	 and
$\underline{\mathfrak{g}}^{\perp_b}=\set{v\in HX \mid b(\xi_X,v)=0,\,\forall \xi_X\in\underline{\mathfrak{g}}}$.
	 Note that \eqref{e-gue200120yydI} is one of the main assumptions in \cite{hsiaohuang}. By \cite[Section 2.5]{hsiaohuang}, we observe that when restricted to~${\mathfrak{g}\times \mathfrak{g}}$, the bilinear form $b$ is nondegenerate on $Y$. Note that in Theorem~\ref{thm:qr=0} we assume that the pseudo-metric $g$ restricted to the orbit through any $m\in \Phi^{-1}(0)$ is non-degenerate. This assumption implies equation \eqref{e-gue200120yydI}. Then, $\mu^{-1}(0)$ is a $d$-codimensional submanifold of $X$. Let $Y:=\mu^{-1}(0)$ and let $HY:=HX\cap TY$.

	 Now, recall that $R$ is the Reeb vector field of the contact form $\omega_0$ and let $\mathcal{L}$ denote the Lie derivative. A vector field $V \colon X \rightarrow TX$ such that $\mathcal{L}_V\omega_0=g \omega_0$ for some smooth function $g \colon X\rightarrow \mathbb{R}$ is called a \textit{contact vector field}. By \cite[Lemma 3.5.14]{mds}, for every function $g \colon X\rightarrow \mathbb{R}$ there exists a unique contact vector field $X_g$ which satisfies
	$\omega_0(X_g)=g $
	 and
	 \[ {\rm d}\omega_0(X_g, V)=({\rm d}g(R)) \alpha(V)-{\rm d}g(V)\]
	 for each vector field $V$ on $X$.
	 Thus, there is a one-to-one correspondence between contact vector fields and smooth functions on $(X, \omega_0)$.
	
	 The analog of the Poisson bracket in contact geometry is
	 \[\{g, h\}:=\omega_0([X_g, X_h])={\rm d}g(X_h)-{\rm d}h(X_g)+{\rm d}\omega_0(X_g, X_h) \]
	 for $f$ and $g$ smooth function on $X$. Now, we are interested in the subspace of $H$-invariant smooth functions on $X$. Thus, for any $S^1$-invariant smooth function $g$ we have ${\rm d}g(R)=R(g)=0$ and~${{\rm d}\omega_0(X_g, V)=-{\rm d}g(V)}$.
	 Eventually, we note that for two $H$-invariant smooth functions $g$ and $h$ on $X$, we obtain
	 \[\{g, h\}=-{\rm d}\omega_0(X_g, X_h)+{\rm d}\omega_0(X_h, X_g)+{\rm d}\omega_0(X_g, X_h)={\rm d}\omega_0(X_h, X_g) .\]
	
	 The CR reduction is the CR manifold $X_G:=\mu^{-1}(0)/G$ \cite{hsiaohuang}. Let $\pi_X \colon \mu^{-1}(0) \rightarrow X_G$ be the projection. Then, $X_G$ is a circle bundle over the symplectic reduction $M_G:=\Phi^{-1}(0)/G$. In Section~\ref{sec:quot}, we explain how to generalize some of the results in~\cite{gs} to the setting of pseudo-Riemannian manifolds. In particular, there is an identification
	 $ M_G\cong M_s/G^{\mathbb{C}}$,
	 where $G^{\mathbb{C}}$ is the complexification of the group $G$ and $M_s$ is the saturation by $G^{\mathbb{C}}$ of $\Phi^{-1}(0)$ in $M$.
	
	 Now, consider a smooth function $\widetilde{f}$ on $X$ invariant under the action of $G$ and $S^1$, it gives rise to a smooth function $\widetilde{f}_{\vert Y}$ on $Y=\mu^{-1}(0)$. Then, $\widetilde{f}_{\vert Y}$ defines a smooth function $f$ on $M_G$. Thus, we define the sup-norm of a $G$-invariant function on $\Phi^{-1}(0)$ to be
	 \[
	 	\big\lVert \widetilde{f}\big\rVert_{\infty} :=\lVert f\rVert_{\infty} := \sup\big\{ \big\lvert \widetilde f(x)\big\rvert \mid x \in \mu^{-1}(0)\big\},
	 \]
	 which is the uniform norm of $f$ on the quotient $M_G$. Vice versa, let $f$ be a smooth $S^1$ invariant function on $X_G$, its lift defines a smooth $G$ and $S^1$ invariant function on $Y$. By \cite[Lemma~5.34\,(b)]{lee}, the extension Lemma for functions on submanifolds, since $Y$ is a smooth properly embedded submanifold of $X$, then there exists a smooth function $\widetilde{f}$ on $X$ whose restriction to~$Y$ coincides with $f$.
	
	 As the last piece of notation, let $\widetilde{f}$ be a $G$-invariant function on $X$, we denote the operator norm of its corresponding level $k$ $G$-invariant Toeplitz operator
	 \begin{equation} \label{eq:opn}\big\lVert T^{G}_{k}\big[\widetilde{f}\big] \big\rVert=\sup \left\{\frac{\lVert T^G_k\big[\widetilde{f}\big]s\rVert}{\lVert s\rVert} \mid s\in H^{G}_k(X) ,\, s\neq 0\right\} .
	 \end{equation}
	
	 In the following theorem, we state an analogue of the main result in \cite{bms}. Parts (c) and (b) of the following theorem are a consequence of the results contained in \cite{gh}, see also \cite{mm,mmt}, where the Bergman kernels and the corresponding Toeplitz operators on symplectic manifolds, for the mixed curvature case, are studied. In Section~\ref{sec:proofoftheorem}, we prove part~(a) by adapting the strategy of~\cite{schl2} to our setting.
	
	 \begin{Theorem} \label{thm:quant} Under the assumptions and notations above. Let $\widetilde{f}$, $\widetilde{g}$ be $G$ and $S^1$ invariant functions on $X$, then
	 	\begin{itemize}\itemsep=0pt
	 		\item[$(a)$] there exists a positive constant $C$, such that
	 		\[\lVert {f}\rVert_{\infty}-\frac{C}{k} \leq \big\lVert T^{G}_{k}\big[\widetilde{f}\big] \big\rVert \leq \lVert {f} \rVert_{\infty}
	 		\]
	 		and in particular $\lim_{k} \big\lVert T^{G}_{k}\big[\widetilde{f}\big] \big\rVert= \lVert {f}\rVert_{\infty}$;
	 		\item[$(b)$] the semi-classical Dirac condition holds
	 		\[ \big\lVert k\mathrm{i} \big[T^{G}_{k}\big[\widetilde{f}\big], T^{G}_{k}\big[\widetilde{g}\big]\big] -T^{G}_{k}\big[\big\{\widetilde{f}, \widetilde{g}\big\}\big]\big\rVert=O\left(\frac{1}{k}\right) ;\]
	 		\item[$(c)$] the semi-classical von Neumann's condition holds
	 		\[\big\lVert T^{G}_{k}\big[\widetilde{f}\big] T^{G}_{k}\big[\widetilde{g}\big] -T^{G}_{k}\big[\widetilde{f}\cdot \widetilde{g}\big]\big\rVert=O\left(\frac{1}{k}\right) .\]
	 	\end{itemize}
	 \end{Theorem}

 	A main consequence of Theorem~\ref{thm:quant} is Theorem~\ref{cor} stated in the introduction. This is shown in Section~\ref{sec:proofcor}.
Finally, note that the Berezin--Toeplitz quantization is also a positive quantization in the following sense, see \cite{lef}. Let $f \in C^{\infty}(X_G)^{S^1}$ be such that $f(x) \geq 0$ for every $x\in X_G$. Then~$T_k^G[f]$ is a non-negative operator, in the sense that $\bigl(T_k^G[f]s, s\bigr)\geq 0$ for every $s \in H_k^G(X)$. Equivalently, by the spectral theorem, a self-adjoint linear operator $A$ is positive if and only if ${\mathrm{Spec}(A)\subset [0, +\infty)}$.
	 On the other hand, the Weyl quantization of a~non-negative function need not be a~non-negative operator.\looseness=-1

	 \section{Preliminaries}\label{s:prelim}
	
	 \subsection{Standard notations} \label{s-ssna}
	 The following notations are adopted through this article: $\mathbb N$ is the set of natural numbers, $\mathbb N_0=\mathbb N\cup\{0\}$, $\mathbb R$ is the set of
	 real numbers, and we denote by ${\mathbb R}_+=\{x\in\mathbb R \mid x>0\}$ and $\overline{\mathbb R}_+=\{x\in\mathbb R\mid x\geq0\}$. Furthermore, we write $\alpha=(\alpha_1,\dots,\alpha_n)\in\mathbb N^n_0$
	 if $\alpha_j\in\mathbb N_0$,
	 $j=1,\dots,n$.
	
	 Let $M$ be a smooth paracompact manifold and let $TM$ and $T^*M$ denote the tangent bundle of $M$
	 and the cotangent bundle of $M$, respectively.
	 The complexified tangent bundle of $M$ and the complexified cotangent bundle of $M$ will be denoted by $TM\otimes \mathbb{C}$
	 and $T^*M\otimes \mathbb{C}$, respectively. Let $\langle \cdot ,\cdot \rangle$ denote the pointwise
	 duality between $TM$ and $T^*M$ and we extend $\langle \cdot ,\cdot \rangle$ bi-linearly to~${ TM\otimes \mathbb{C}\times T^*M\otimes \mathbb{C}}$.
	 Let $B$ be a smooth vector bundle over $M$ whose fiber at $x\in M$ will be denoted by $B_x$.
	 If $E$ is another vector bundle over a smooth paracompact manifold $N$, we write
	 $B\boxtimes E^*$ to denote the vector bundle over $M\times N$ with fiber over $(x, y)\in M\times N$
	 consisting of the linear maps from $E_y$ to $B_x$.
	
	 Let $Y\subset M$ be an open set, the spaces of distribution sections of $B$ over $Y$ and smooth sections of $B$ over $Y$ will be denoted by $ D'(Y, B)$ and ${C}^\infty(Y, B)$, respectively.
	 Let $ E'(Y, B)$ be the subspace of $ D'(Y, B)$ whose elements have compact support in $Y$.

	 Now, we recall the Schwartz kernel theorem.
	 Let $B$ and $E$ be smooth vector
	 bundles over paracompact orientable smooth manifolds $M$ and $N$, respectively, equipped with smooth densities of integration. If
	$A\colon {C}^\infty_0(N, E)\To D'(M, B)$
	 is continuous, we write $A(x, y)$ to denote the distribution kernel of $A$.
	 The following two statements are equivalent
	 \begin{enumerate}\itemsep=0pt
	 	\item[(1)] $A$ is continuous: $ E'(N, E)\To{C}^\infty(M, B)$,
	 	\item[(2)] the distributional kernel $A(x,y)$ lies in ${C}^\infty(M\times N,B\boxtimes E^*)$.
	 \end{enumerate}
	 If $A$ is continuous, we say that $A$ is smoothing on $M \times N$. Let
	 \[A,\hat A\colon\ {C}^\infty_0(N, E)\To{D}'(M, B)\] be continuous operators, then we write
$
	 	\mbox{$A\equiv \hat A$ (on $M\times N$)}
$
	 if $A-\hat A$ is a smoothing operator. If $M=N$, we simply write ``on $M$".
	 We say that $A$ is properly supported if the restrictions of the two projections
	 $(x,y)\mapsto x$, $(x,y)\mapsto y$ to ${\rm supp }(A(x,y))$
	 are proper.
	
	 Eventually, let $H(x,y)\in D'(M\times N, B\boxtimes E^*) $
	 then we write $H$ to denote the unique
	 continuous operator $ C^\infty_0(N,E)\To D'(M,B)$ with distribution kernel $H(x,y)$ and we identify $H$ with~$H(x,y)$.
	
	 \subsection{Operators and symbols} \label{sec:crmanif}
	
	 Let $\bigl(X, T^{1,0}X\bigr)$ be a compact, connected and orientable CR manifold of dimension $2n+1$, $n\geq 1$.
	
	 Let $D$ be an open set of $X$. Recall that $\Omega^{0,q}(D)$ denote the space of smooth sections of $T^{*0,q}X$ over $D$ and let $\Omega^{0,q}_c(D)$ be the subspace of $\Omega^{0,q}(D)$ whose elements have compact support in~$D$. Let
$
	 \ddbar_b\colon\Omega^{0,q}(X)\To\Omega^{0,q+1}(X)
$
	 be the tangential Cauchy--Riemann operator. Let ${\rm d}v(x)$ be the volume form on $X$ induced by the Hermitian metric $\langle \cdot | \cdot \rangle$.
	 The natural global $L^2$ inner product~$( \cdot | \cdot )$ on $\Omega^{0,q}(X)$
	 induced by~${\rm d}v(x)$ and $\langle \cdot | \cdot \rangle$ is given by
	 \[
	 ( u | v ):=\int_X\langle u(x) | v(x) \rangle {\rm d}v(x) ,\qquad u,v\in\Omega^{0,q}(X) .
	 \]
	 Recall that \smash{$L^2_{(0,q)}(X)$} denote
	 the completion of $\Omega^{0,q}(X)$ with respect to $( \cdot | \cdot )$.
	 We extend $( \cdot | \cdot )$ to~\smash{$L^2_{(0,q)}(X)$}
	 in the standard way. For \smash{$f\in L^2_{(0,q)}(X)$}, we denote $\norm{f}^2:=( f | f )$.
	 We extend
	 $\ddbar_{b}$ to~\smash{$L^2_{(0,r)}(X)$}, $r=0,1,\dots,n$, by
	 \[
	 \ddbar_{b}\colon \ \operatorname{Dom}\ddbar_{b}\subset L^2_{(0,r)}(X)\To L^2_{(0,r+1)}(X) ,
	 \]
	 where \[\operatorname{Dom}\ddbar_{b}:=\big\{u\in L^2_{(0,r)}(X)\mid \ddbar_{b}u\in L^2_{(0,r+1)}(X)\big\}\] and, for any $u\in L^2_{(0,r)}(X)$, $\ddbar_{b} u$ is defined in the sense of distributions.
	 We also write
	 \[
	 \ol{\pr}^{*}_{b}\colon \ \operatorname{Dom}\ol{\pr}^{*}_{b}\subset L^2_{(0,r+1)}(X)\To L^2_{(0,r)}(X)
	 \]
	 to denote the Hilbert space adjoint of $\ddbar_{b}$ in the $L^2$ space with respect to $( \cdot | \cdot )$.
	 Let $\Box^{q}_{b}$ denote the Gaffney extension of the Kohn Laplacian whose domain $\operatorname{Dom}\Box^q_{b}$ is given by
	 \begin{equation*}
	 		\big\{s\in L^2_{(0,q)}(X)\mid
	 		s\in\operatorname{Dom}\ddbar_{b}\cap\operatorname{Dom}\ol{\pr}^{*}_{b},\,
	 		\ddbar_{b}s\in\operatorname{Dom}\ol{\pr}^{*}_{b}, \, \ol{\pr}^{*}_{b}s\in\operatorname{Dom}\ddbar_{b}\big\}
	 \end{equation*}
	 and
$
	 		\Box^{q}_{b}s=\ddbar_{b}\ol{\pr}^{*}_{b}s+\ol{\pr}^{*}_{b}\ddbar_{b}s
 $ for $s\in \operatorname{Dom}\Box^q_{b} $.
	 Let
	 \[
	 	S^{(q)}\colon\ L^2_{(0,q)}(X)\To\operatorname{Ker}\Box^q_b
	 \]
	 be the orthogonal projection with respect to the $L^2$ inner product $( \cdot | \cdot )$ and let
	 \[
	 S^{(q)}(x,y)\in{D}'\bigl(X\times X,T^{*0,q}X\boxtimes\bigl(T^{*0,q}X\bigr)^*\bigr)
	 \]
	 denote the distribution kernel of $S^{(q)}$.
	
	 Now, we recall H\"ormander symbol space. Let $D\subset X$ be a local coordinate patch with local coordinates $x=(x_1,\dots,x_{2n+1})$.
	
	 \begin{Definition}\label{d-gue140221a}
	 	For $m\in\Real$,
$
S^m_{1,0}\bigl(D\times D\times\mathbb{R}_+,T^{*0,q}X\boxtimes\bigl(T^{*0,q}X\bigr)^*\bigr)
$
	 	is the space of all
\[
a(x,y,t)\in{C}^\infty\bigl(D\times D\times\mathbb{R}_+,T^{*0,q}X\boxtimes\bigl(T^{*0,q}X\bigr)^*\bigr)
\]
	 	such that, for all compact $K\Subset D\times D$ and all $\alpha, \beta\in\mathbb N^{2n+1}_0$, $\gamma\in\mathbb N_0$,
	 	there is a constant $C_{\alpha,\beta,\gamma}>0$ such that
	 	\[\big|\pr^\alpha_x\pr^\beta_y\pr^\gamma_t a(x,y,t)\big|\leq C_{\alpha,\beta,\gamma}(1+\abs{t})^{m-\gamma} \qquad
	 	\text{for every}\ (x,y,t)\in K\times\Real_+, \ t\geq1.\]
	 	Put
$S^{-\infty}\bigl(D\times D\times\mathbb{R}_+,T^{*0,q}X\boxtimes\bigl(T^{*0,q}X\bigr)^*\bigr)$
	 	to be
	 	\[\bigcap_{m\in\Real}S^m_{1,0}\bigl(D\times D\times\mathbb{R}_+,T^{*0,q}X\boxtimes\bigl(T^{*0,q}X\bigr)^*\bigr).
	 	\]
	 	
	 	Now, let \[a_j\in S^{m_j}_{1,0}\bigl(D\times D\times\mathbb{R}_+,T^{*0,q}X\boxtimes\bigl(T^{*0,q}X\bigr)^*\bigr) ,\]
	 	$j=0,1,2,\dots$ with $m_j\To-\infty$, as $j\To\infty$.
	 	Then there exists \[a\in S^{m_0}_{1,0}\bigl(D\times D\times\mathbb{R}_+,T^{*0,q}X\boxtimes\bigl(T^{*0,q}X\bigr)^*\bigr)\]
	 	unique modulo $S^{-\infty}$ such that
	 	\[a-\sum^{k-1}_{j=0}a_j\in S^{m_k}_{1,0}\bigl(D\times D\times\mathbb{R}_+,T^{*0,q}X\boxtimes\bigl(T^{*0,q}X\bigr)^*\bigr)\]
	 	for $k=0,1,2,\dots$.
	 	
	 	If $a$ and $a_j$ have the properties above, we write $a\sim\sum^{\infty}_{j=0}a_j$ in
	 	\[S^{m_0}_{1,0}\bigl(D\times D\times\mathbb{R}_+,T^{*0,q}X\boxtimes\bigl(T^{*0,q}X\bigr)^*\bigr) .\]
	 	Furthermore, we write
	 	\[
	 	s(x, y, t)\in S^{m}_{{\rm cl }}\bigl(D\times D\times\mathbb{R}_+,T^{*0,q}X\boxtimes\bigl(T^{*0,q}X\bigr)^*\bigr)
	 	\]
	 	if \[s(x, y, t)\in S^{m}_{1,0}\bigl(D\times D\times\mathbb{R}_+,T^{*0,q}X\boxtimes\bigl(T^{*0,q}X\bigr)^*\bigr)\] and
	 	\[
	 		s(x, y, t)\sim\sum^\infty_{j=0}s^j(x, y)t^{m-j}\qquad \text{in} \ S^{m}_{1, 0}
	 		\bigl(D\times D\times\mathbb{R}_+ ,T^{*0,q}X\boxtimes\bigl(T^{*0,q}X\bigr)^*\bigr) ,
	 	\]
	 	where
	 	\[
	 		s^j(x, y)\in{C}^\infty\bigl(D\times D,T^{*0,q}X\boxtimes\bigl(T^{*0,q}X\bigr)^*\bigr),\qquad j\in\N_0 .
	 	\]
	 	Sometimes, we simply write $S^{m}_{1,0}$ to denote
$
S^m_{1,0}\bigl(D\times D\times\mathbb{R}_+,T^{*0,q}X\boxtimes\bigl(T^{*0,q}X\bigr)^*\bigr)$,
 where $m\in\mathbb R\cup\set{-\infty}$.
	 \end{Definition}
	
	 Let $E$ be a smooth
	 vector bundle over $X$. Let $m\in\Real$, $0\leq\rho,\delta\leq1$. Let $S^m_{\rho,\delta}(X,E)$
	 denote the H\"{o}rmander symbol space on $X$ with values in $E$ of order $m$ type $(\rho,\delta)$
	 and let $S^m_{{\rm cl }}(X,E)$
	 denote the space of classical symbols on $X$ with values in $E$ of order $m$.
	 For $a\in S^m_{{\rm cl }}(X,E)$, we~write ${\rm Op }(a)$ to denote the pseudodifferential operator on $X$ of order $m$ from sections of $E$ to sections of $E$ with full symbol~$a$.
	
	 Let $D\subset X$ be an open set. Let
$L^m_{{\rm cl }}\bigl(D,T^{*0,q}X\boxtimes \bigl(T^{*0,q}X\bigr)^*\bigr)$
	 denote the space of classical
	 pseudodifferential operators on $D$ of order $m$ from sections of
	 $T^{*0,q}X$ to sections of $T^{*0,q}X$ respectively. Let $P\in L^m_{{\rm cl }}\bigl(D,T^{*0,q}X\boxtimes \bigl(T^{*0,q}X\bigr)^*\bigr) $ and we write $\sigma_P$ and $\sigma^0_P$ to denote the full symbol of $P$ and the principal symbol of $P$, respectively.
 	
 	 \subsection{Toeplitz operators}
	 \label{sec:toep}
	
	 In this subsection, we briefly recall some results from~\cite{gh}. Here, the Reeb vector field $R$ is transversal to the space $\underline{\mathfrak{g}}$ at every point $p\in\mu^{-1}(0)$,
${\rm e}^{{\rm i}\theta}\cdot g\circ x=g\circ {\rm e}^{{\rm i}\theta}\cdot x$, for every $ x\in X$, $\theta\in[0,2\pi]$, $ g\in G$,
	 and the action of $G$ and $S^1$ is free near $\mu^{-1}(0)$.
		
	For every smooth function $g\in{C}^\infty(X)^G$ invariant under the action of $S^1$, the $k$-th $G$-invariant Toeplitz operator is given by
	 \[
	 	T^G_{k}[g]:=S^{G}_{k}\circ g\cdot\circ S^{G}_{k}\colon\ L^2_{(0,n_-)}(X)\To L^2_{(0,n_-),k}(X)^G.
	 \]
	 The main interest for us is the case when $g=\tilde{f}$ and $f$ is in \smash{$C^{\infty}(X_G)^{S^1}$}. For the following theorem see in \cite[Theorem 1.8]{hsiaohuang} and \cite{gh}.
	
	 \begin{Theorem} \label{thm:kfourierszego}
	 	Let $f\in{C}^\infty(X)^G$ and suppose $q= n_-$. Let $D$ be an open set in $X$ such that the intersection $\mu^{-1}(0)\cap D= \varnothing$. Then, $T^{G}_{k}[f]\equiv O(k^{-\infty})$ on $D$.
	 	
	 	Let $p\in \mu^{-1}(0)$ and let $U$ a local neighborhood of $p$. Then, if $q= n_-$, for every fixed~${y\in U}$, we consider $T_{k}^{G}[f](x,y)$ as a $k$-dependent smooth function in $x$. There exist local coordinates $(x_1,\dots ,x_{2n+1})$ on $U$ such that
	 	\[T^G_{k}[f](x, y)= {\rm e}^{{\rm i} k \Psi(x, y)} b(x, y, k)+O(k^{-\infty}) \]
	 	for every $x, y\in U$. For an explicit expression of the phase function $\Psi$ we refer to {\rm\cite{gh}}, the symbol satisfies
	 	\[b(x, y, m)\in S^{n-d/2}_{\mathrm{loc}}\bigl(1, U\times U, T^{*0,q}X\boxtimes\bigl(T^{*0,q}X\bigr)^*\bigr) \]
	 	and the leading term of $b(x,y,m)$ along the diagonal at $x\in\mu^{-1}(0)$ is given by
	 	\[b_0(x, x):= 2^{d-1}\frac{1}{V_{{\rm eff }}(x)}\pi^{-n-1+\frac{d}{2}}\abs{\det R_x}^{-\frac{1}{2}}\abs{\det{L}_{x}}\tau_{x,n_-},
 \]
	 	where
$V_{{\rm eff }}(x):=\int_{O_x}{\rm dV}_{O_x}$,
	 	$O_x$ is the orbit through $x$, and $R_x$ is the linear map
	 	\[
	 		R_x\colon\ \underline{\mathfrak{g}}_x\To\underline{\mathfrak{g}}_x,\qquad
	 		u\To R_xu,\qquad\langle R_xu | v \rangle=\langle {\rm d}\omega_0(x) , Ju\wedge v \rangle.
	 	\]
	 \end{Theorem}
	
 Now we recall the definition of the operator $\hat{H}$, which appears in the description of the $C_j$'s, see \cite{gh}. It is convenient to introduce a class of operator which looks microlocally like the Szeg\H{o} kernel. Here we recall the definition from \cite{gh} for $q=n_-$ and action free case.

 \begin{Definition}\label{d-gue210803yyd}
 	Let $H\colon \Omega^{0,n_-}(X)\To\Omega^{0,n_-}(X)$ be a continuous operator with distribution kernel \[H(x,y)\in{D}'\bigl(X\times X,T^{*0,n_-}X\boxtimes\bigl(T^{*0,n_-}X\bigr)^*\bigr) .\]
 	We say that $H$ is a complex Fourier integral operator of $G$-invariant Szeg\H{o} type of order ${k\in\mathbb Z}$ if~$H$ is smoothing away $\mu^{-1}(0)$ and $H$ has the following microlocal expression. Let $p\in \mu^{-1}(0)$, and let $D$ a local neighborhood of $p$ with local coordinates $(x_1,\dots x_{2n+1})$. Then, the distributional kernel of $H$ satisfies
 	\[H(x,y)\equiv \int^\infty_0 {\rm e}^{{\rm i} t \Phi_-(x,y)}a_-(x, y, t) {\rm d}t\]
 	on $D$ where \[a_-\in S^{k-\frac{d}{2}}_{{\rm cl }}\bigl(D\times D\times\mathbb{R}_+,T^{*0,n_-}X\boxtimes\bigl(T^{*0,q}X\bigr)^*\bigr) ,\] $a_+=0$ if $q\neq n_+$, where $\Phi_-$ is the phase of the $G$-invariant Szeg\H{o} kernel. Furthermore, we write~${\sigma^0_{H,-}(x,y)}$ to denote the leading term of the expansion of $a_{-}$.
 	
 	Let $\Psi_k(X)^G$ denote the space of all complex Fourier integral operators of Szeg\H{o} type of order~$k$.
 \end{Definition}

Here we adapt the discussion before \cite[Theorem 4.9]{gh} to $G$-action case. Let
\[
	\hat R:=\frac{1}{2}S^{G}(-{\rm i}R+(-{\rm i}R)^*)S^{G}\colon\ \Omega^{0,n_-}(X)\To\Omega^{0,n_-}(X),
\]
where $(-{\rm i}R)^*$ is the adjoint of $-iR$ with respect to $( \cdot | \cdot )$.
Then, we note that $\hat R\in\Psi_{n+1}(X)^G$ with $\sigma^0_{\hat R,-}(x,x)\neq0$, for every $x\in X$.
Let $\hat H\in\Psi_{n-1}(X)^G$ with
\begin{equation}\label{e-gue210706ycdh}
		S^G\hat H=\hat H=\hat HS^G ,\qquad
		\hat R\hat H\equiv \hat H\hat R\equiv S^G .
\end{equation}
Note that $\hat H$ is uniquely determined by \eqref{e-gue210706ycdh}, up to some smoothing operators. The operator~$\hat H$ can always be found by setting $\hat H= S^G \tilde{H} S^G$ for a given pseudodifferential operator $\tilde{H}$. In fact, the first equation in \eqref{e-gue210706ycdh} is always satisfied. Furthermore, the second equation can be solved by an easy application of the stationary phase formula.

Now, we recall the $G$-invariant version of \cite[Theorem 4.9]{gh}.

\begin{Theorem}\label{t-gue210706ycdg}
	Let $q=n_-$.
	Let $f, g\in{C}^\infty(X)$, then we have
	\[
		T^G[f]\circ T^G[g]-\sum^N_{j=0}\hat H^j T^G\big[{\hat C_{j}(f,g)}\big]\in\Psi_{n-N-1}^G(X)
	\]
	for every $N\in\mathbb N_0$, where $\hat C_{j}(f,g)\in{C}^\infty(X)$, $\hat C_{j}$ is a universal bidifferential operator of order $\leq 2j$, $j=0,1,\dots$, and
$\hat C_{0}(f,g)=f\cdot g$, $
			\hat C_{1}(f,g)-\hat C_{1}(g,f)= {\rm i} \{f,g\}$.
\end{Theorem}

Eventually, we recall the following theorem from \cite{gh}.
	
	 \begin{Theorem} \label{thm:compositionFouriercomponents}
	 	With the same assumptions as above, let $f, g\in{C}^\infty(X)^G$. Let $q=n_-$. Then, as $k\gg1$,
	 	\[
	 		\norm{T^G_{f,k}\circ T^G_{g,k}-T^G_{g,k}\circ T^G_{f,k}-\frac{1}{k} T^G_{ {\rm i} \{f,g\},k}}=O\bigl(k^{-2}\bigr),\]
	 	and
	 	\[
	 		\norm{T^G_{f,k}\circ T^G_{g,k}-\sum^N_{j=0}k^{-j} T^G_{C_j(f,g),k}}=O\bigl(k^{-N-1}\bigr)
	 	\]
	 	in $L^2$ operator norm, for every $N\in\mathbb N$, where $C_j(f,g)\in{C}^\infty(X)^G$,
	 	$C_j$ is a universal bidifferential operator of order $\leq 2j$, $j=0,1,\dots$, and
$C_{0}(f,g)=f\cdot g $, $
	 			C_{1}(f,g)-C_{1}(g,f)= {\rm i} \{f,g\}$.
	 	
	 	Moreover, the star product
	 	\[
	 		f*g= \sum_{j=0}^{+\infty} C_{j}(f,g) \nu^{-j} ,
	 	\]
	 	$f, g\in{C}^\infty(X)^G_0$, is associative.
	 \end{Theorem}

\section{Berezin transform}
	
\subsection{Vector valued reproducing kernel}
	
Here, we recall some definitions concerning reproducing kernel Hilbert spaces for vector valued functions, we refer to \cite{bm,pe} (see also \cite[Section~2]{cvt}).
	
\begin{Definition}Let $X$ be a manifold and $E$ be a vector bundle on $X$. A $E$-valued reproducing kernel Hilbert space on $X$ is a Hilbert space $H$ such that
		\begin{itemize}\itemsep=0pt
			\item[(1)] the elements of $H$ are sections of the bundle $E\rightarrow X$,
			\item[(2)] for all $x\in X$ there exists a positive constant $C_x$ such that
			$\lVert s(x)\rVert_{E_x} \leq C_x \lVert s\rVert_H $ for each~${s\in H}$.
		\end{itemize}
	\end{Definition}
	
	In this article, we shall consider the case where $X$ is CR manifold satisfying the assumption of the Introduction and $H$ is the Hilbert space $H^{G}(X)$. Now, let $q=n_-$, $x\in X$, and $s\in H^{G}_k(X)$ then $s(x)$ is in the vector space \smash{$T^{* (0,q)}_xX$} which is endowed with the Hermitian metric $\langle \cdot \vert\cdot \rangle$. The Hilbert space \smash{$H^{G}_k(X)$} is a vector valued reproducing kernel Hilbert space on $X$ with inner product $( \cdot\vert\cdot )$. The evaluation map
	\[\mathrm{ev}_x \colon\ H^{G}_k(X) \rightarrow T^{* (0,q)}X,\qquad \mathrm{ev}_x(s)=s(x)\]
	is a bounded operator and the reproducing kernel associated with $H^{G}_k(X)$ is $S^{G}$. The kernel~$S^{G}$ reproduces the value of a section $s\in H^{G}_k(X)$ at point $x\in X$. Indeed, for all $x\in X$ and~\smash{$v\in T^{* (0,q)}_xX$}
	$\mathrm{ev}_x^* v = S^{G}_k(\cdot, x) v $
	so that
	$\langle s(x)\vert v \rangle = \bigl(s, S^{G}_k(\cdot, x)v\bigr)$.
	Thus, for each $x\in X$ and~\smash{$v \in T^{* (0,q)}_xX$} the $(0,q)$-forms $s_{x,v}=S^{G}_k(\cdot, x) v$ is called the coherent state (of level $k$) associated to $x$ and $v$.
	
	\begin{Definition} \label{def:cov}
		The covariant Berezin symbol \smash{$\sigma^{(k)}_k(A)$} of an equivariant operator of Szeg\H{o} type~${A\in \Psi^G_k(X)}$ is
		\[\sigma_v^{(k)}(A) \colon\ X \rightarrow \mathbb{C},\qquad x\mapsto \sigma_v(A)(x):=\frac{(s_{x,v}\vert As_{x,v})}{(s_{x,v}\vert s_{x,v})} \]
		for a given fixed $v\in T^{*(0,q)}_xX$.
	\end{Definition}

	\begin{Remark}
		The definition of coherent states goes back to Berezin. A coordinate independent version and extensions to line bundles were given by Rawnsley \cite{r}, we mainly refer to \cite{schl2} where the coherent vectors are parameterized by the elements of $L^*\setminus 0$, where $L$ is a positive line bundle of a given Hodge manifold. In the following we shall generalize the coherent states in our setting.
	\end{Remark}
	
	Let $f$ be a $G$ invariant function on $\mu^{-1}(0)$, for any extension $\widetilde{f}$ on $X$, we assign to it its Toeplitz operator $T^{G}_{k}\big[\widetilde{f}\big]$ and then assign to it the covariant symbol, as in Definition \ref{def:cov} \[\sigma_v^{(k)}\bigl(T^{G}_{k}\big[\widetilde{f}\big]\bigr) \colon\ X \rightarrow \mathbb{C}\] which is again an element of $C^{\infty}(X)^{G}$ and thus it defines a circle invariant function on the CR reduction $X_G$.

 \subsection{Berezin transforms for Toeplitz operators}
	
	Recall that $\pi\colon \mu^{-1}(0)\rightarrow X_G$ is the projection. Given a smooth function $f\in C^{\infty}(X_G)$ invariant under the circle action, we define the Berezin transform of level~$k$.
	
 \begin{Definition} The map
\smash{$C^{\infty}(X_G)^{S^1} \rightarrow C^{\infty}(X_G)^{S^1} $}, \smash{$ f \mapsto I_k^G[f]:=\sigma^{(k)}_v\bigl(T^{G}_{k}[f]\bigr)$}
 		is called the Berezin transform of level $k$.
\end{Definition}

In particular, we are interested in studying $I_k^G[f]$ on a neighborhood $U$ of $p$ in $\mu^{-1}(0)$.
 	
 	\begin{Theorem} \label{thm:main}
 		Let $p\in \mu^{-1}(0)$ and let $U$ a local neighborhood of $p$. Then, if $q= n_-$, there exist local coordinates $(x_1,\dots, x_{2n+1})$ on $U$ such that the Berezin transform $I^{G}_k[f]$ evaluated at the point $x\in U$ has a complete asymptotic expansion in decreasing integer powers of $k$
 		\[I^G_k[f](x)\sim \sum_{j=0}^{+\infty}I_j[f](x) \frac{1}{k^j} , \]
 		as $k$ goes to infinity, where $I_j \colon C^{\infty}(X_G)^{S^1}\rightarrow C^{\infty}(X_G)^{S^1}$ are maps such that $I_0[f]=f$.
 	\end{Theorem}
 		\begin{proof}
 		Let $f\in C^{\infty}(X_G)^{S^1}$ and let us first study the numerator of $\sigma^{(k)}_v\bigl(T^{G}_{k}\big[\widetilde{f}\big]\bigr)$ which is
 		\begin{align*} \bigl(S^{G}_k(\cdot, x) v\big\vert T^{G}_{k}\big[\widetilde{f}\big] S^{(q)}_k(\cdot, x) v\bigr) =\int_X v^{\dagger} S^{(q)}_k(x, y) \widetilde{f}(y) S^{(q)}_k(y, x) v \mathrm{dV}_X(y) .
 		\end{align*}
 		The denominator is obtain by taking $f\equiv 1$. Thus, by Proposition \ref{thm:kfourierszego}, we have a complete asymptotic expansion in decreasing integer power of $k$, and we obtain
 		\[\bigl(S^{(q)}_k(\cdot, x) v\big \vert T^{(q)}_{k}\big[\widetilde{f}\big] S^{(q)}_k(\cdot, x) v\bigr) \sim k^d v^{\dagger}\sigma^{0}_{T^G [\widetilde{f} ]}(x, x) v +O\bigl(k^{d-1}\bigr)\]
 		where
 		\[\sigma^0_{T^G [\widetilde{f} ]}(x,x)=\frac{1}{2}\pi^{-n-1}\abs{\det L_x} \widetilde{f}(x) \tau_{x,n_-} ,\]
 		for every $x\in \mu^{-1}(0)$. Thus, taking the quotient we get
 		$I_k^G\big[\widetilde{f}\big](x)=\widetilde{f}(x)+O\bigl(k^{d-1}\bigr) $
 		for each~${x\in \mu^{-1}(0)}$.
 	\end{proof}
 	
 	Theorem~\ref{thm:main} is consequence of the asymptotic expansion of Toeplitz operators for $(0,q)$-forms and the stationary phase formula. Here, we shall mainly refer to \cite{gh} where we study Toeplitz operators on CR manifolds, see also \cite{hsiao,mm} for the asymptotic behavior of the Bergman kernel for $(0,q)$-forms, and references therein.
 	
 	\section{Quantization commutes with reduction}
 	\subsection{Quotients and complexification of groups}
 	\label{sec:quot}
 	
 	Let $G$ be a compact connected Lie group and let $\mathfrak{g}$ be its Lie algebra. Here, we recall some results proved in \cite[Section 4]{gs} and we explain how to extend it to pseudo K\"ahler manifolds. There exists a unique connected complex Lie group, $G^{\mathbb{C}}$, with the following two properties:
 	\begin{itemize}\itemsep=0pt
 		\item its Lie algebra is $\mathfrak{g}\oplus \sqrt{-1} \mathfrak{g}$,
 		\item $G$ is a maximal compact subgroup of $G^{\mathbb{C}}$.
 	\end{itemize}

 	Now, we shall prove some properties of $G$-actions on pseudo-K\"ahler manifolds. A polarization of $X$ is an integrable Lagrangian subbundle, $F$, of $TM\otimes \mathbb{C}$. The next lemma, an analogue of~\cite[Lemma 4.3]{gs}, will be used to reconcile this definition with the standard one.
 	
 	\begin{Lemma} \label{lem:gs} Let $V$ be a $($real$)$ symplectic vector space with symplectic form, $\Omega$. Let $E$ be a~Lagrangian subspace of $V\otimes \mathbb{C}$ Then there exists a unique linear mapping $K \colon V \rightarrow V$ such that
 	\begin{itemize}	\itemsep=0pt\samepage
 	\item[$(i)$] $K^2 = -I$,
 	\item[$(ii)$] $E=\big\{v+\sqrt{-1} Kv, v\in V\big\}$,
 	\item[$(iii)$] $\Omega(Kv, Kw)=\Omega(v, w)$,
 	\item[$(iv)$] The quadratic form $B(v, w)= \Omega(v, K w)$ is symmetric and non-degenerate.
 	\end{itemize}
 	\end{Lemma}
 	
 	Let $(M, \omega)$ be a symplectic manifold and $F$ a polarization. By the lemma we get for each~${m\in M}$ a mapping $K=J_m \colon T_mM\rightarrow T_mM$ with the properties (i), (ii), and (iii) and a~quadratic form, $B=g_m$, on $M$. Furthermore, $J$ defines an almost-complex structure on $M$ and $g$ a~pseudo-Riemannian structure. The integrability of $F$ implies that the almost-complex structure is complex. Therefore, the quadruple $(M,J,g, \omega)$ is a pseudo-K\"ahler manifold in the usual sense.
 	
 	Let $(M,\omega)$ be a compact pseudo-K\"ahler manifold and $G$ a compact connected Lie group which acts on $M$, preserving $F$. We shall prove Theorem~\ref{thm:4.4} below, which is an analogue of~\cite[Theorem~4.4]{gs} and in fact the first part of the proof is identical. However, note that \cite[Theorem~4.4]{gs} relies on the fact that the group of analytic diffeomorphisms of $M$ which preserve~$F$ is a~finite-dimensional Lie group, see \cite{k}. Here, we prove that the former result still holds true in the setting of pseudo-Riemannian manifolds (in Theorem~\ref{thm:4.4}, we assume that $G^{\mathbb{C}}$ is simply connected, for the general case we refer to \cite[Theorem~4.4]{gs}).
 	
 	Let $\mathrm{Iso}(X, g)$ be the group of isometries from $(M, g)$ onto itself, that is $f \in \mathrm{Iso}(X, g)$ if and only if $f$ is a smooth diffeomorphism and $f^*g=g$. Let $\bigl(M, T^{1,0}M\bigr)$ be a connected complex manifold and denote by $\mathrm{Iso}(M, g)$ the group of isometries on $M$ with respect to some Riemannian metric $g$. Let $\mathrm{Aut}_J(M)$ be the group of complex automorphisms on $M$, that is $f\in\mathrm{Aut}_J(M)$ if and only if $f \colon M\rightarrow M$ is a smooth diffeomorphism satisfying ${\rm d}f\bigl(T^{1,0}M\bigr) \subseteq T^{1,0}M$.
 	
 	\begin{Theorem} \label{thm:4.4} Let $G^{\mathbb{C}}$ be simply-connected. The action of $G$ can be canonically extended to an action of $G^{\mathbb{C}}$, preserving the pseudo-K\"ahler structure of $M$.
 	\end{Theorem}
 	\begin{proof}
 		Let
$\tau \colon \mathfrak{g}^{\mathbb{C}} \rightarrow $ (real) vector fields on $M$
 		be the mapping, \smash{$\xi^{(1)}+\sqrt{-1} \xi^{(2)} \mapsto \xi^{(1)}_M+J \xi^{(2)}_M$}
 		where $\xi^{(1)}, \xi^{(2)}\in \mathfrak{g}$. By \cite[equations~(4.2) and (4.3)]{gs}, $\tau$ is a morphism of Lie algebras. Moreover, by~(4.2), if $\eta \in \mathfrak{g}^{\mathbb{C}}$, $\tau(\eta)$ is a vector field preserving $F$. By Lemma \ref{lem:finliegroup} below, $\mathrm{Iso}(M, g)\cap \mathrm{Aut}_J(M)$ is a~(finite-dimensional) Lie group; therefore, if $G^{\mathbb{C}}$ is simply-connected, $\tau$ can be extended uniquely to a morphism of Lie groups.
 	\end{proof}

 	\begin{Lemma} \label{lem:finliegroup} The group $\mathrm{Iso}(M, g)\cap \mathrm{Aut}_J(M)$ is a Lie group.
 	\end{Lemma}
 	\begin{proof}
 		By \cite{k,ph}, we have that $\mathrm{Iso}(M, g)$ is a finite-dimensional Lie group, since $(M, g)$ is a~pseudo-Riemannian manifold. Recall that	a subgroup of a Lie group is itself a Lie group if it is a closed subgroup.
 		
 		Now, consider the subgroup $\mathrm{Iso}(M, g, J)\subset \mathrm{Iso}(M, g)$, consisting of those diffeomorphisms that preserve both the metric
 		$g$ and the complex structure $J$. The preservation of $g$ and	$J$ by a~diffeomorphism are smooth conditions. Thus, we consider the set $\mathrm{Iso}(M, g, J)$ as the solution set of these equations within the group $\mathrm{Iso}(M, g)$. If a sequence of diffeomorphisms in $\mathrm{Iso}(M, g, J)$ converges to a diffeomorphism in the topology of
 		$\mathrm{Iso}(M, g)$ (which is the compact-open topology), the limit will also preserve $g$ and
 		$J$.
 		
 		Thus, $\mathrm{Iso}(M, g, J)$ is a closed subgroup of the Lie group $\mathrm{Iso}(M, g)$, it is itself a Lie group.
 	\end{proof}

 	Recall now that $(M, \omega)$ is the complex pseudo-K\"ahler manifold obtained by quotienting out by the circle action on $X$, $M:=X/S^1$. The CR action of $G$ on $X$ descends to a complex and Hamiltonian action on $M$ with moment map $\Phi \colon M\rightarrow \mathfrak{g}^*$. Thus, we form the reduced space~${M_G:= \Phi^{-1}(0)/G}$.
 	Let $M_s$ be the saturation of $\Phi^{-1}(0)$ with respect to $G^{\mathbb{C}}$,
 	\[M_s:=\big\{g\circ x \mid x\in \Phi^{-1}(0),\, g\in G^{\mathbb{C}} \big\} , \]
 	the points of $M_s$ are called stable points. In the same way as in \cite[Theorem~4.5]{gs}, one can prove the following theorem. Here, we shall retrace the proof again since we use the assumption that the pseudo-metric $g$ restricted to the orbit through any $m\in \Phi^{-1}(0)$ is non-degenerate.
 	
 	\begin{Theorem}
 		Assume that the pseudo-metric $g$ restricted to the orbit through any $m\in \Phi^{-1}(0)$ is non-degenerate and the action of $G$ on $\Phi^{-1}(0)$ is free.
 		The set $M_s$ is an open subset of $M$ and $G^{\mathbb{C}}$ acts freely on it and $M_G$ can be represented as the quotient space $M_G:=M_s/G^{\mathbb{C}}$.
 	\end{Theorem}
 	\begin{proof}
 		Let $V$ be a real symplectic vector space and $F$ a	Lagrangian subspace of the complexification $V\otimes \mathbb{C}$. Let $J$ and $B$ as in Lemma \ref{lem:gs}.

 		Since the origin in $\mathfrak{g}$ is a regular value of $\Phi$, then $\Phi^{-1}(0)$ is a $G$-invariant co-isotropic submanifold of $M$. Moreover, the action of $G$ on $\Phi^{-1}(0)$ is locally free and the orbits of $G$ are the leaves of the null-foliation.
 		
 		Recall that $b$ is the nondegenerate bilinear form on $HX$ such that
$b(\cdot , \cdot) = {\rm d}\omega_0(\cdot , J\cdot)$.
 		Let us denote $\mu^{-1}(0)$ by $Y$. There is a natural projection $\pi_Y \colon Y\rightarrow \Phi^{-1}(0)$. Recall that we assume that~$0$ is a regular value of $\mu$, and that the action of $G$ on $\mu^{-1}(0)$ is free. Furthermore, by hypothesis the bilinear form~$b$ is nondegenerate on $Y$. Then,
 		$\pi_Y(\underline{\mathfrak{g}}_x)=T_{\pi(x)}\Phi^{-1}(0) $
 		and
 		\[\big\{{\xi}_M(\pi(x))\in T_{\pi(x)}M \mid \xi\in \sqrt{-1} \mathfrak{g} \big\} \]
 		is a complementary space to $T_{\pi(x)}\Phi^{-1}(0)$ in $T_{\pi(x)}M$. This shows that $M_s$ contains an
 		open neighborhood, $U$, of $\Phi^{-1}(0)$. Since
 		\[M_s=\bigcup_{g\in G^{\mathbb{C}} } g U,\]
 		then $M_s$ is itself open. Thus, the stabilizer algebra of $x$ in $\mathfrak{g}^{\mathbb{C}}$ is zero; so the action of $G^{\mathbb{C}}$ on $M_s$ is locally free. To show that $G^{\mathbb{C}}$ acts freely on~$M_s$ one can proceed as in \cite[the second part of Theorem 4.5, p.~527]{gs}.
 	\end{proof}

 	\subsection[Proof of Theorem \ref{thm:qr=0}]{Proof of Theorem~\ref{thm:qr=0}}
 	\label{sec:proofof1.2}
 	
 	The action of $G$ on the line bundle, $L$, can be canonically extended to an action of $G^{\mathbb{C}}$ on $L$. The proof of this fact is identical with the proof of \cite[Theorem 5.1]{gs}, and we will omit it. Here, we just recall the infinitesimal action of an element $\eta=\sqrt{-1} \xi\in \sqrt{-1}\mathfrak{g}$ on a harmonic $L$-valued $(0,n_-)$-form $s$
 	\begin{equation} \label{eq:5.3}
 		-\nabla_{\eta_M} s-2\pi \langle\Phi(\cdot), {\xi}\rangle s .
 	\end{equation}
 	
 	Let $s$ be a harmonic $L$-valued $(0,n_-)$-form, and let $(s, s)(m)$ be the norm of $s$. By definition~$(s,s)$ is a non-negative real-valued function. By assumption, $( \,,\, )$ is invariant with respect to parallel transport; so for all $\eta=\sqrt{-1} \xi\in \sqrt{-1}\mathfrak{g}$,
$\eta_M (s,s)=(\nabla_{\eta_M}s, s)+(s, \nabla_{\eta_M}s)$.
 	Suppose now that $s$ is $G^{\mathbb{C}}$-invariant. Then, by \eqref{eq:5.3},
 	$\nabla_{\eta_M} s=-2\pi \langle\Phi(\cdot), {\xi}\rangle s$
 	so
 	$\eta_M\langle s, s \rangle=-4\pi \langle\Phi(\cdot), {\xi}\rangle \langle s, s\rangle $.
 	
 	Now let $\mathcal{Q}(M)$ be the space of harmonic $L$-valued $(0,n_-)$-forms and let $\mathcal{Q}(M_s)$ be the space of harmonic $L$-valued $(0,n_-)$-forms over $M_s$. Let $\mathcal{Q}(M)^G$ band $\mathcal{Q}(M_s)^G$. be the set of G-fixed vectors in these two spaces. The following theorem follows as in \cite[Theorem 5.2]{gs}.
 	
 	\begin{Theorem} \label{thm:5.2}
 		The canonical mapping $\mathcal{Q}(M_s)^G \rightarrow \mathcal{Q}(M_G)$ is bijective.
 	\end{Theorem}
 	
 	It is clear that the restriction mapping $\mathfrak{r} \colon \mathcal{Q}(M)^G \rightarrow \mathcal{Q}(M_s)^G$ is injective; so, by Theorem~\ref{thm:5.2}, to prove ``quantization commutes with reduction'', it is enough to prove that $\mathfrak{r}$ is surjective. Now, we shall need the following theorem.
 	
 	\begin{Theorem} \label{thm:openanddense}
 		The set $M\setminus M_s$ is contained in a complex sub-manifold of $M$ of complex codimension greater than one.
 	\end{Theorem}

 	This theorem is an analogue of \cite[Theorem 5.7]{gs} for pseudo-Riemannian manifolds. Note that here the existence theorem \cite[Theorem 5.6]{gs} is not true in general for the pseudo-Riemannian case: we have to replace the spaces of smooth sections of $L^k$ with space of $L^k$-valued $(0, n_-)$-forms, see Theorem~\ref{thm:2}. Then, Theorem~\ref{thm:openanddense} follows by combining Theorems \ref{thm:1} and \ref{thm:2} below.
 	
 	\begin{Theorem} \label{thm:1}
 		Let $s$ be a $G$-invariant $L^k$-valued $(0, n_-)$-forms on $M$. Then every $m\in M$ fulfilling $s(m)\neq 0$ lies in $M_s$.
 	\end{Theorem}
 	\begin{proof}
 		See the proofs of \cite[Theorems 5.3 and 5.4]{gs}, they hold true in our setting up to replacing the space of $G$-invariant holomorphic sections of $L^k$ with the space of $G$-invariant $L^k$-valued $(0, n_-)$-forms.
 	\end{proof}

 	Before stating Theorem~\ref{thm:2}, we recall the following result from \cite{hsiaohuang}. Although it can be deduced from the theory of Toeplitz operators discussed in the previous sections, here we present Theorem~\ref{t-gue220914yyd} in terms of the Bergman kernel, rather than the Szeg\H{o} kernel. As we explained in Section~\ref{sec:for}, we identify the $G$-invariant Bergman projector \smash{$P^{(q),G}_k$} of $L^k$ with the $k$-th Fourier component of the $G$-invariant Szeg\H{o} kernel \smash{$S^{(q)}_{G,k}$}.
 	
 	Let us premise a further piece of notation. We will identify the curvature form $R^L$ with the Hermitian matrix
 	\[\dot{R}^L\in\mathcal{C}^{\infty}\bigl(M,\mathrm{End}\bigl(T^{1,0}M\bigr)\bigr) , \qquad \big\langle R^L(z) , U\wedge\overline{V} \big\rangle =\big\langle \dot{R}^L(z)U | V \big\rangle,\]
 	for every $U$ and $V$ in $T^{1,0}_zM$, $z\in M$. We denote by $W$ the subbundle of rank $q=n_-$ of $T^{1,0}M$ generated by the eigenvectors corresponding to negative eigenvalues of $\dot{R}^L$. Then, $\det \ol{W}^*:=\Lambda^q\ol{W}^*$ is a rank one subbundle, where $\ol{W}^*$ is the dual bundle of the complex conjugate bundle of $W$ and $\Lambda^q\ol{W}^*$ is the vector space of all finite sums $v_1\wedge\cdots\wedge v_q$, $v_1,\dots,v_q\in\ol{W}^*$. We denote by
 	$P_{\det \ol{W}^*}$ the orthogonal projection from $T^{*0,q}M$ onto $\det \ol{W}^*$.
 	
 	Let $s$ and $s_1$ be local holomorphic trivializing sections of $L$ defined on open sets $D\subset M$ and~${D_1\subset M}$, respectively, $\abs{s}^2_{h^L}={\rm e}^{-2\phi}$, $\abs{s_1}^2_{h^L}={\rm e}^{-2\phi_1}$, $\phi\in\mathcal{C}^\infty(D,\mathbb R)$, $\phi_1\in\mathcal{C}^\infty(D_1,\mathbb R)$ which can be assumed to be $G$-invariant. The localization of \smash{$P^{(q),G}_k$} with respect to $s$, $s_1$ is given by
 	\[
 		P^{(q),G}_{k,s,s_1}\colon\ \Omega^{0,q}_c(D)\To\Omega^{0,q}(D_1),\qquad
 		u\To s^{-k}_1{\rm e}^{-k\phi_1}\bigl(P^{(q),G}_k\bigl(s^k{\rm e}^{k\phi}u\bigr)\bigr).
\]
 	Let \[P^{(q),G}_{k,s,s_1}(x,y)\in\mathcal{C}^\infty\bigl(D_1\times D,T^{*0,q}M\boxtimes\bigl(T^{*0,q}M\bigr)^*\bigr)\] be the distribution kernel of \smash{$P^{(q),G}_{k,s,s_1}$}. When $D=D_1$, $s=s_1$, we write \smash{$P^{(q),G}_{k,s}:=P^{(q),G}_{k,s,s}$},
 	\smash{$P^{(q),G}_{k,s}(x,y):=P^{(q),G}_{k,s,s}(x,y)$}.
 	
 	\begin{Theorem}\label{t-gue220914yyd}
 		With the notations and assumptions used above and recall that we let $q=n_-$. Let $s$ be a local holomorphic trivializing section of $L$ defined on an open set $D\subset M$, $\abs{s}^2_{h^L}={\rm e}^{-2\phi}$.
 		
 		Let $D$ be an open set of $X$ with $D\cap \Phi^{-1}(0) =\varnothing$. Then, as $k\rightarrow +\infty$, \smash{$P^{(q),G}_{k,s}=O(k^{-\infty})$} on~$D$.
 		
 		Let $m \in \Phi^{-1}(0)$ and let $U$ be an open set of $m$. Then, as $k\rightarrow +\infty$, in local coordinates defined in $U$ $($as introduced in {\rm\cite{hsiaohuang})},
 		\[
 			P^{(q),G}_{k,s}(x,y)={\rm e}^{{\rm i}k\Psi(x,y)}b(x,y,k)+O(k^{-\infty})\qquad\text{on} \ D\times D,
 		\]
 		where \[b\in S^{n-d/2}_{{\rm cl }}\bigl(1;D\times D,T^{*0,q}M\boxtimes\bigl(T^{*0,q}M\bigr)^*\bigr) ,\qquad b(x,y,k)\sim\sum_{j=0}b_j(x,y)k^{n-j}\] in $S^{n-d/2}_{{\rm cl }}\bigl(1;D\times D,T^{*0,q}M\boxtimes\bigl(T^{*0,q}M\bigr)^*\bigr) $,
 		and \[b_j\in\mathcal{C}^\infty\bigl(D\times D,T^{*0,q}M\boxtimes\bigl(T^{*0,q}M\bigr)^*\bigr) ,\qquad b_0=\mathbf{b_0}(m) P_{\det \overline{W}^*}\]
 with $\mathbf{b_0}(m)> 0$ and $j=0,1,\dots$,
 		and we refer to {\rm\cite[\emph{Theorem} 1.8]{hsiaohuang}} for the properties of the phase function $\Psi\in\mathcal{C}^\infty(D\times D)$.
 	\end{Theorem}

 	\begin{Theorem}\label{thm:2}
 		If the set $\Phi^{-1}(0)$ is non-empty and zero is a regular value of $\Phi$, then for some~$k$, there exists a global nonzero holomorphic $G$-invariant $L^k$-valued $(0, n_-)$-forms.
 	\end{Theorem}
	\begin{proof}
		Let $S_0, \dots, S_{d_k}$ be an orthonormal basis of the vector space \smash{$H^{q}\bigl(M,L^{\otimes k}\bigr)^G$}. In the standard situation, when the line bundle is positive, we consider $q=0$. Let $m \in \Phi^{-1}(0)$; since all the components transform by the same scalar under a change of frame and since
			\[P^{(0),G}_k(m, m) = \sum_{j=0}^{d_k} |S_j(m)|^2_{h^{L^{\otimes k}}}=O\bigl(k^{n-d/2}\bigr) , \]
		it is easy to see that the sections $S_j$ do not share common zeros.
		
		In the pseudo-Riemannian setting, we assume $q=n_-$. Let $m \in \Phi^{-1}(0)$, in view of Theorem~\ref{t-gue220914yyd}, we have
		\begin{equation} \label{eq:expan}
			P_k^{(q)}(m,m)= \mathbf{b_0}(m) P_{\det \overline{W}^*} k^{n-d/2} + O\bigl(k^{n-d/2-1}\bigr) ,
		\end{equation}
		with $\mathbf{b_0}(m)> 0$. By \eqref{eq:expan}, it follows that the sections $S_j$ do not share common zeros. Since~$M$ is compact, there exists $k_0$ such that \smash{$P^{(n_-),G}_k(m, m)$} for all $k\geq k_0$ and $m\in M$. Thus, by equation~\eqref{eq:expan}, there do not exist $m\in M$ such that for each $s$ in $\mathcal{Q}_k(M)^G$, $s(m)=0$ for all~${k>0}$. In particular, since $\Phi^{-1}(0)$ is non-empty, then for some $k$, there exists a global non-zero holomorphic $G$-invariant $L^k$-valued $(0, n_-)$-forms.
	\end{proof}

	 ``Quantization commutes with reduction'' follows as in \cite{gs} by combining \cite[Theorem 5.7]{gs} (the analogous result here is Theorem~\ref{thm:openanddense}) and the following analogue of \cite[Theorem 5.8]{gs}.
	
	\begin{Theorem} \label{thm:final}
		Let $s$ be a vector in $\mathcal{Q}_k(M)^G$. Then $\langle s, s \rangle$ is bounded and takes its maximum value on $\Phi^{-1}(0)$.
	\end{Theorem}
	\begin{proof}
		Let $m$ be a point of $M_s$. Then $m=g m_0$ with $m_0\in \Phi^{-1}(0)$ and $g\in G^{\mathbb{C}}$. By the Cartan decomposition of $G^{\mathbb{C}}$, we have $G^{\mathbb{C}}=P G$, where $G$ is a maximal compact subgroup of $G^{\mathbb{C}}$ and the Lie algebra of $P$ is $\sqrt{-1} \mathfrak{g}$. Thus we write $g={\rm e}^{\eta} k$ with $\eta =\sqrt{-1} \xi\in \sqrt{-1} \mathfrak{g}$ and $k\in G$. Replacing $m_0$ with $k m_0$, we can assume that $m={\rm e}^{\eta} m_0$.
		
		Consider the curve $\gamma \colon (-\infty, +\infty)\rightarrow M$, $ t\mapsto \gamma(t)={\rm e}^{t \eta} m_0$. Hence, by the formula
$\eta_M\langle s, s\rangle=-4\pi \langle\Phi(\cdot), \xi\rangle \langle s, s \rangle$
		we have, along $\gamma(t)$,
$\frac{{\rm d}}{{\rm d}t}\langle s, s\rangle=-4\pi \langle\Phi(\cdot), \xi\rangle \langle s, s \rangle$.
		As in \cite[Lemma 4.7]{gs}, we note that $\eta_M$ is the gradient vector field associated with the function~${\Phi^{\xi}:=\langle\Phi(\cdot), \xi\rangle}$. Thus $\Phi^{\xi}$ is strictly increasing along $\gamma(t)$, so it is positive
		for $t>0$ and negative for $t<0$. Therefore, $\langle s, s \rangle$ has a unique maximum mum at $t=0$.
	\end{proof}
 	
 	Finally, we note that Theorem~\ref{thm:final} implies the surjectivity of $\mathfrak{r}$. Indeed,
 	if $m$ is a point of~$M$, then we can find a neighborhood~$U$ of $m$ in $M$ and a non-zero holomorphic $G$-invariant $L^k$-valued $(0, n_-)$-forms, $s_0 \colon U\rightarrow L^k\otimes T^{*0,n_-}M $. Then $s =f s_0$ on $U\cap M_s$, $f$ being a bounded holomorphic function. Since $(M\setminus M_s)\cap U$ is of complex codimension greater than one in $U$, the singularity of $f$ at $m$ is removable. Thus $s$ extends to a holomorphic section of $L^k\otimes T^{*0,n_-}M $ over all of $M$.
 	 	
\section[Proof of Theorem \ref{thm:quant}]{Proof of Theorem~\ref{thm:quant}}
	\label{sec:proofoftheorem}
	
	Let $\widetilde f\in C^{\infty}(X)$ be a $G$ and $S^1$ smooth invariant function on $X$. First, we prove that
	\begin{equation} \label{eq:ineq}
		\big\lVert I_{k}^G[f] \big\rVert_{\infty} \leq \big\rVert T^{G}_k\big[\widetilde{f}\big]\big\lVert \leq \lVert f\rVert_{\infty} .
	\end{equation}

	Recall that given two $G$ invariant $(0, n_-)$ forms $s$ and $t$ in $H(X)^G_k$, we denote their inner product by $(s\vert t)$. Now, using the Cauchy--Schwartz inequality, we obtain that for each element~${x\in \mu^{-1}(0)}$
	\begin{align*}
	\big\lVert I^{G}_k[f](x) \big\rVert^2_{\infty}&= \frac{\big\vert\bigl(S^{G}_k(\cdot, x) v\big\vert T^{G}_{k}[f] S^{G}_k(\cdot, x) v\bigr)\big\vert^2}{\bigl(S^{G}_k(\cdot, x) v\big\vert S^{G}_k(\cdot, x) v\bigr)^2} \\
	&\leq \frac{\bigl(T^{G}_{k}[f] S^{G}_k(\cdot, x) v\big\vert T^{G}_{k}[f] S^{G}_k(\cdot, x) v\bigr)}{\bigl(S^{G}_k(\cdot, x) v\big\vert S^{G}_k(\cdot, x) v\bigr)}
	\leq \rVert T^{G}_k[f]\lVert^2,
\end{align*}
 	where the last inequality follows by the definition of the operator norm, see \eqref{eq:opn}. Taking the $\sup$ both sides over $\mu^{-1}(0)$, it shows the first inequality in \eqref{eq:ineq}.
 	
 	For the second inequality, we note that for $s \in H^{G}_k(X)$,
 	\begin{align*}
 \frac{\lVert T^{G}_k[f] s \rVert^2}{\lVert s \rVert^2}&=\frac{\bigl(T^{G}_{k}[f] s \big\vert T^{G}_{k}[f] s\bigr)}{\int_X \langle s, s\rangle \mathrm{dV}_X} = \frac{\lVert f\rVert^2_{\infty} \bigl(S^{G}_{k}[f] s \big\vert S^{G}_{k}[f] s\bigr)}{\int_{X} \langle s, s\rangle \mathrm{dV}_{X}}
 		\leq \lVert f\rVert^2_{\infty} .
 	\end{align*}
 	Hence, it proves the second inequality \eqref{eq:ineq}.
 	
 	Now, we are ready to prove the theorem. Since $X$ is compact, then $\mu^{-1}(0)$ is closed and hence compact. Choose $x_e\in \mu^{-1}(0)$ a point with $\big\lvert \widetilde{f}(x_e)\big\rvert=\lVert f \rVert_{\infty}$. Since the Berezin transform has as a leading term the identity, we have
\smash{$\big\lvert \bigl(I^{G}_k[f]\bigr)(x_e) -\widetilde{f}(x_e) \big\rvert \leq \frac{C}{k} $}
 	for a suitable $C>0$. Thus,
 	\[ \lVert f \rVert_{\infty}-\frac{C}{k}= \big\lvert \widetilde{f}(x_e)\big\rvert -\frac{C}{k}\leq \big\lvert \bigl(I^{G}_k[f]\bigr)(x_e)\big\rvert \leq \big\lVert I^{G}_k[f]\big\rVert_{\infty} .\]
 	Putting \eqref{eq:ineq} in this last inequality, we get part (a) of Theorem~\ref{thm:quant}.
 	
 	Parts (b) and~(c) of Theorem~\ref{thm:quant} are consequence of Theorem~\ref{thm:compositionFouriercomponents}.
 	
 	\section[Proof of Theorem \ref{cor}]{Proof of Theorem~\ref{cor}}
 	\label{sec:proofcor}
 	
 	It is an easy consequence of \cite{gh} that the star product defined by Toeplitz operators is of \textit{Wick type}.
 	
 	Let $f$ and $g$ be two $G$ and $S^1$ invariant smooth function on $X_G$. For ease of notation, we write $T^{G}_k[f]$ for $T^{G}_k\big[\widetilde{f}\big]$. It is clear by Theorem~\ref{thm:quant} and \cite{gh} that $\star$ is a well-defined associative star product. In fact, by part (b) of Theorem~\ref{thm:quant}, we get
 	\[\big\lVert T^{G}_k[{\{f, g\}-i(C_1(f,g)-C_1(g,f))}] \big\rVert=O\left(\frac{1}{k}\right), \]
 	and taking the limit for $k\rightarrow+\infty$ and using part (a) we get
$\lVert \{f, g\}-{\rm i}(C_1(f,g)-C_1(g,f))\rVert_{\infty}=0 $
 	and thus we get part (b). Similarly, by Theorem~\ref{thm:quant}\,(c) and~(a), we get that $C_0(g,h)=g\cdot h$. In a similar way the associativity property is proved. 	
 	
 	Now, we prove the uniqueness. Let $C_j$ and $\tilde{C}_j$ be two systems of bi-differential operators inducing star products and both fulfilling the asymptotic condition \eqref{eq:cond}. We show that $C_j = \tilde{C}_j$, for all $j\in \mathbb{N}_0$, hence the induced star products coincide. Note $C_0=\tilde{C}_0$ by hypothesis. By induction, assume that $C_j=\tilde{C}_j$ for each $j\leq N-2$. By the results in \cite{gh}, we have
 	\[T_k^G[f]\circ T_k^G[g]\sim\sum_{j=0}^{N-1}k^{-j}T_k^G[C_j(f, g)] +O\bigl(k^{-N}\bigr)\]
 	and
 	\[T_k^G[f]\circ T_k^G[g]\sim\sum_{j=0}^{N-1}k^{-j}T_k^G\big[\tilde{C}_j(f, g)\big]+O\bigl(k^{-N}\bigr) , \]
 	where $\sim$ stand for ``as the same asymptotic as''. For every $f, g\in C^{\infty}(X_G)^{S^1}$, by the inductive hypothesis and subtracting these two expression, we obtain
 	\[\left\lVert \frac{1}{k^{N-1}} T^G_k\big[{C}_{N-1}(f, g)-\tilde{C}_{N-1}(f, g)\big]\right\rVert\leq \frac{K}{k^{N}} . \]
 	Eventually, taking the limit on both sides as $k$ goes to infinity and applying Theorem~\ref{thm:quant}\,(a), we obtain
 	$C_{N-1}(f, g)=\tilde{C}_{N-1}(f, g) $.
 	
 	Now we recall that
 	\[A_N=\hat{R}^lT^G[f]T^G[g]-\sum_{j=0}^{N-1} \hat{R}^{N-j}T^G[C_j(f, g)] \]
 	which is an operator of Szeg\H{o} type of order zero by Theorem~\ref{t-gue210706ycdg} and it is invariant under the circle action. Thus, its symbol is a $G$ and $S^1$ invariant function on $X_G$ which, by restriction on~$\mu^{-1}(0)$, defines a $S^1$-invariant function on $X_G$. By definition, it is the next element $C_N(f, g)$ in the star-product.
 	
 	The unit of the algebra \smash{$C^{\infty}(X_G)^{S^1}$}, the constant function $1$, will also be the unit in the star-product. In fact, it is equivalent to
 	\begin{equation} \label{eq:cj}
 		C_j(1, f)=C_j(f, 1)=0
 		\end{equation}
 	for $j\geq 1$ and for every $f\in C^{\infty}(X_G)^{S^1}$. First, by Theorem~\ref{thm:kfourierszego} note that
 	$ C_0(1, g)=C_0(1, g)=g$
 	for every $k$. 	Now, if we put $g=1$,{\samepage
\[
A_1= \hat{R}T^G[f]T^G[g]-\hat{R}T^G[f\cdot g]= 0,
\]
 	and hence the symbol of~$A_1$ vanishes. By induction, the claim \eqref{eq:cj} follows.}
 	
 	Let us now prove that the Berezin--Toeplitz star product fulfills parity. Considering the formal parameter to be real, $\nu =\overline{\nu}$, this is equivalent to
 	\begin{equation} \label{eq:ck}
 		\overline{C_k(f, g)}=C_k(\overline{g}, \overline{f})
 	\end{equation}
 	for $k\geq 0$. Note that, as an easy consequence of the definition of Berezin--Toeplitz star product, we have
 	$T_k^G[f]^{\dagger}=T_k^G[\overline{f}]$.
 	The star product $\overline{g}\star\overline{f}$ is given by the asymptotic expansion of the composition
 	\begin{align*}
 		T_k^G[\overline{g}]\cdot T_k^G\big[\overline{f}\big]&=T_k^G[g]^{\dag}\cdot T_k^G[f]^{\dag} =\bigl(T_k^G[{f}]\cdot T_k^G[g]\bigr)^{\dag} \\
 		&\sim \sum_{j=0}^{+\infty} k^{-j} T_k^G[C_j(f, g)]^{\dag}
 		\sim \sum_{j=0}^{+\infty} k^{-j} T_k^G\big[\overline{C_j(f, g)}\big] .
 	\end{align*}
 	Thus, the claim \eqref{eq:ck} follows.
 	
 	By \cite{gh}, see also \cite{gh2}, the trace $\mathrm{Tr}_k$ on $H(X)^G_k$ admits a complete asymptotic expansion in decreasing integer power of $k$
 	\[\mathrm{Tr}_k\bigl(T^G_k[f] \bigr)\sim k^d\left(\sum_{k=0}^{+\infty} k^{-j} \tau_j(f)\right) ,\]
 	with $\tau_j(f)\in \mathbb{C}$. It induces a $\mathbb{C}[[\nu]]$-linear map
$\mathrm{Tr} \colon {C}^{\infty}(X_G)^{S^1}[[\nu]]\rightarrow \nu^{-n} \mathbb{C}[[\nu]]$
 	such that
 	\[\mathrm{Tr}(f):= \sum_{j=0}^{+\infty}\nu^{j-n} \tau_j(f) , \]
 	where for $f \in C^{\infty}(M)$ the $\tau_j(f)$ are given by the asymptotic expansion above and for arbitrary elements by $\mathbb{C}[[\nu]]$-linear extension, $\nu = k^{-1}$.
 	Now, we shall prove
 	\begin{equation} \label{eq:tr}
 		\mathrm{Tr}(f\star g)=\mathrm{Tr}(g\star f) .
 	\end{equation}
 	
 	By $\mathbb{C}[[\nu]]$-linearity, we prove \eqref{eq:tr} for $f$ and $g$ in \smash{${C}^{\infty}(X_G)^{S^1}$}. Note that $f\star g - g \star f$ is given by the asymptotic expansion of
 	$ T^G_k[f] T^G_k[g]- T^G_k[g] T^G_k[f]$.
 	Hence the trace of $f\star g - g \star f$ is given by the expansion of
 	$\mathrm{Tr}_k\bigl( T^G_k[f] T^G_k[g]- T^G_k[g] T^G_k[f] \bigr) $.
 	But for every $k$ this vanishes and thus we get \eqref{eq:tr}.

\subsection*{Acknowledgements}
We are indebted to the referees for various interesting comments and for suggesting several improvements. The author thanks Chin-Yu Hsiao and Herrmann Hendrik for their valuable conversations. The author is a member of GNSAGA, part of the Istituto Nazionale di Alta Matematica, and expresses gratitude to the group for its support.
 	
\pdfbookmark[1]{References}{ref}
\LastPageEnding

\end{document}